\RequirePackage{fix-cm}
\documentclass{svjour3}                     
\usepackage{amsmath,amsfonts}

\smartqed  

\usepackage[table]{xcolor}
\usepackage{graphicx}
\usepackage[framemethod=tikz]{mdframed}
\usepackage{epstopdf}
\usepackage{hyperref}
\usepackage{subcaption}
\usepackage{multirow}
\usepackage{booktabs}
\usepackage{algpseudocode}
\usepackage{todonotes}

\begin{document}

\title{Fully Adaptive Zeroth-Order Method for Minimizing Functions with Compressible Gradients\thanks{G. N. Grapiglia was partially supported by FRS-FNRS, Belgium (Grant CDR J.0081.23).}
}

\titlerunning{ZORO-FA}        

\author{Geovani N. Grapiglia \and Daniel McKenzie}

           \institute{G.N. Grapiglia \at
              Université catholique de Louvain, ICTEAM/INMA, Avenue Georges Lema\^{i}tre, 4-6/ L4.05.01, B-1348, Louvain-la-Neuve, Belgium\\               geovani.grapiglia@uclouvain.be.
}

\institute{D. McKenzie \at 
Colorado School of Mines, Applied Mathematics and Statistics Department, Colorado, USA\\ dmckenzie@mines.edu.
}
\date{Received: date / Accepted: date}

\maketitle

\begin{abstract} We propose an adaptive zeroth-order method for minimizing differentiable functions with $L$-Lipschitz continuous gradients. The method is designed to take advantage of the eventual compressibility of the gradient of the objective function, but it does not require knowledge of the approximate sparsity level $s$ or the Lipschitz constant $L$ of the gradient. We show that the new method performs no more than $\mathcal{O}\left(n^{2}\epsilon^{-2}\right)$ function evaluations to find an $\epsilon$-approximate stationary point of an objective function with $n$ variables. Assuming additionally that the gradients of the objective function are compressible, we obtain an improved complexity bound of $\mathcal{O}\left(s\log\left(n\right)\epsilon^{-2}\right)$ function evaluations, which holds with high probability. Preliminary numerical results illustrate the efficiency of the proposed method and demonstrate that it can significantly outperform its non-adaptive counterpart.

\keywords{derivative-free optimization \and black-box optimization\and zeroth-order optimization \and worst-case complexity\quad compressible gradients}

\end{abstract}

\section{Introduction}
In this work we consider the unconstrained minimization of a possibly nonconvex differentiable function $f:\mathbb{R}^{n}\to\mathbb{R}$. Most well-known iterative algorithms for this type of problem are {\em gradient-based}, meaning at each iteration they require at least one evaluation of the gradient of $f(\,\cdot\,)$. However, in many problems, it is impossible to compute the gradient of the underlying objective function. Examples include simulation-based optimization \cite{Nakamura2017}, reinforcement learning \cite{salimans2017evolution,mania2018simple,choromanski2020provably} and hyperparameter tuning \cite{bergstra2012random,tett2022calibrating}. These settings create a demand for {\em zeroth-order optimization algorithms}, which are algorithms that rely only on evaluations of the objective function $f(\,\cdot\,)$. Recently, problems requiring zeroth-order methods have emerged in which $n$, the dimension of the problem, is of order $10^5$ or greater. In particular, we highlight many recent works applying zeroth-order methods to fine-tune Large Language Models (LLMs) \cite{malladi2024fine,zhang2024revisiting,liu2024sparse,chen2023deepzero}. For problems of this scale, classical zeroth-order algorithms can struggle to make progress as they generically require $\mathcal{O}(n)$ function evaluations (``queries'') per iteration. \textcolor{black}{For problems of moderate to large scale, deterministic zeroth-order algorithms often struggle to make progress, as they typically require $\mathcal{O}(n)$ function evaluations (or ``queries'') per iteration. To mitigate this issue, many randomized zeroth-order methods have been proposed in recent years, with cheaper per-iteration cost in terms of function evaluations. Some of these methods mimic classical zeroth-order methods, but applied in a low-dimensional subspace \cite{Nesterov,Bergou,Kozak,Cartis,Roberts}. Others aim to exploit `low-dimensional' structure within $f$ \cite{Bandeira,cartis2022dimensionality,cartis2023bound,wang2016bayesian,Yue2023zeroth}.}

Here, we focus on methods that exploit \textcolor{black}{a certain kind of low dimensional structure, namely} (approximate) gradient sparsity \cite{wang2018stochastic,balasubramanian2018zeroth,cai2021zeroth,cai2022one,cai2022zeroth,qiu2024gradient}, particularly the {\tt ZORO}-type algorithms \cite{cai2021zeroth,cai2022zeroth} \textcolor{black}{which robustly approximate $\nabla f(x)$ using tools from compressed sensing}\footnote{\textcolor{black}{In \cite{Bandeira}, compressed sensing techniques have been used to exploit (approximate) sparsity of $\nabla^{2}f(x)$ in the context of derivative-free trust-region methods.}}. One adverse consequence of this fusion of optimization and signal processing techniques is that {\tt ZORO}-type algorithms have an unusually large number of hyperparameters. In this paper, we address this issue by proposing a {\em fully adaptive} version of {\tt ZORO}, which we call {\tt ZORO-FA}. The proposed method is designed to take advantage of the eventual compressibility of the gradient of the objective function, in which case a suitable gradient approximation can be obtained with high probability using $\mathcal{O}\left(s\log(n)\right)$ function evaluations. Specifically, at each iteration, the new method first attempts to use these cheap gradient approximations, and only when they fail \textcolor{black}{a forward-finite difference gradient is computed} using $\mathcal{O}\left(n\right)$ function evaluations. Moreover, {\tt ZORO-FA} is endowed with a novel line-search procedure for selecting multiple hyperparameters at once, adapted from that proposed in \cite{grapiglia2024quadratic}. We show that the new method takes at most $\mathcal{O}\left(n^{2}\epsilon^{-2}\right)$ function evaluations to find $\epsilon$-approximate stationary points. Assuming that the gradients of the objective function are compressible with sparsity level $s$, we obtain an improved complexity bound of $\mathcal{O}\left(s\log\left(n\right)\epsilon^{-2}\right)$ function evaluations, which holds with high probability. Our preliminary numerical results indicate that {\tt ZORO-FA} requires significantly fewer function evaluations than {\tt ZORO} to find points with similar function values, and that the improvement factor increases as the dimensionality of the problems grows.\\

\noindent\textbf{Contents.} The rest of the paper is laid out as follows. In Section 2 we present the key auxiliary results about our gradient approximation and the corresponding variant of the descent lemma. In Section 3, we present the new zeroth-order method and establish a worst-case oracle complexity bound. Finally, in Section 4 we report some numerical results.
 
\section{Auxiliary Results}

In what follows we will consider the following assumptions:

\begin{itemize}
\item[\textbf{A1}] $\nabla f:\mathbb{R}^{n}\to\mathbb{R}^{n}$ is $L$-Lipschitz continuous.
\item[\textbf{A2}] $f$ is bounded below, that is there exists $f_{\mathrm{low}}\in\mathbb{R}$ such that 
\begin{equation*}
    f(x) \geq f_{\mathrm{low}} \quad \text{ for all } x \in \mathbb{R}^n.
\end{equation*}
\item[\textbf{A3}] There exists $p\in (0,1)$ such that, for all $x\in\mathbb{R}^{n}$, $\nabla f(x)$ is $p$-compressible with magnitude $\|\nabla f(x)\|$, i.e., 
\begin{equation*}
\left|\nabla f(x)\right|_{(j)}\leq\dfrac{\|\nabla f(x)\|}{j^{1/p}},\quad j=1,\ldots,n,
\end{equation*}
where $|u|_{(j)}$ denotes the \textcolor{black}{$j$-th} largest-in-magnitude component of $u\in\mathbb{R}^{n}$.
\end{itemize}

 \noindent Given $x\in\mathbb{R}^{n}$ and $s\in\left\{1,\ldots,n\right\}$, let 
 \begin{equation*}
\left[\nabla f(x)\right]_{(s)}=\arg\min\left\{\|v-\nabla f(x)\|_{2}\,:\,\|v\|_{0}\leq s\right\},
 \end{equation*}
 where $\|v\|_0$ denotes the cardinality of the support of $v$. Note that $\left[\nabla f(x)\right]_{(s)}$ is the $s$-sparse vector formed by the $s$ largest (in mo\-du\-lus) entries of $\nabla f(x)$.

 \begin{lemma}
Suppose that A3 holds and let $s\in\left\{1,\ldots,n\right\}$. Then, for all $x\in\mathbb{R}^{n}$ we have
\begin{equation*}
\|\nabla f(x)-\left[\nabla f(x)\right]_{(s)}\|_{1}\leq\dfrac{\|\nabla f(x)\|_{2}}{\left(\frac{1}{p}-1\right)}s^{1-1/p},
\end{equation*}
and
\begin{equation*}
\|\nabla f(x)-\left[\nabla f(x)\right]_{(s)}\|_{2}\leq\dfrac{\|\nabla f(x)\|_{2}}{\left(\frac{2}{p}-1\right)^{\frac{1}{2}}}s^{1/2-1/p}.
\end{equation*}
 \label{lem:X}
 \end{lemma}
 \begin{proof}
See Section 2.5 in \cite{needell2009cosamp}.\qed.
 \end{proof}

\noindent We say that a matrix $Z\in\mathbb{R}^{m\times n}$ has the $4s$-Restricted Isometry Property ($4s$-RIP) when, for every $v\in\mathbb{R}^{n}$ with $\|v\|_{0}\leq 4s$ we have
\begin{equation*}
\left(1-\delta_{4s}(Z)\right)\|v\|_{2}^{2}\leq\|Zv\|_{2}^{2}\leq \left(1+\delta_{4s}(Z)\right)\|v\|_{2}^{2}
\end{equation*}
for some constant $\delta_{4s}(Z)\in (0,1)$. 

\begin{lemma}(Proposition 3.5 in \cite{needell2009cosamp}) If $Z\in\mathbb{R}^{m\times n}$ satisfies the $4s$-RIP for some $s\in\left\{1,\ldots,n\right\}$, then for any $v\in\mathbb{R}^{n}$ we have
\begin{equation*}
\|Zv\|_{2}\leq\sqrt{1+\delta_{4s}(Z)}\left(\|v\|_{2}+\frac{1}{\sqrt{s}}\|v\|_{1}\right).
\end{equation*}
\label{lem:Y}
\end{lemma}
Consider the function $c_{0}:\mathbb{R}\to\mathbb{R}$ given by
\begin{equation}
c_{0}(a)=\dfrac{a^{2}}{4}-\frac{a^{3}}{6}.
\label{eq:auxiliar}
\end{equation}

\begin{lemma} (Theorem 5.2 in \cite{baraniuk2008simple}) Given $s\in\mathbb{N}$ such that $m:=\lceil b s\ln(n)\rceil<n$, let $Z\in\mathbb{R}^{m\times n}$ be defined as
\begin{equation*}
Z=\frac{1}{\sqrt{m}}\left[z_{1}\ldots z_{m}\right]^{T},
\end{equation*}
where $z_{i}\in\mathbb{R}^{n}$ are Rademacher random vectors\footnote{That is, $[z_i]_j = \pm 1$ with equal probability.}. Given $\delta\in (0,1)$, if 
\begin{equation}
b>\textcolor{black}{c_{1}(\delta)\equiv}\dfrac{4\left(1+\dfrac{1+\ln(\frac{12}{\delta})}{\ln(\frac{n}{4s})}\right)}{c_{0}(\frac{\delta}{2})},
\label{eq:revision1}
\end{equation}
for $c_{0}(\,\cdot\,)$ defined in (\ref{eq:auxiliar}), then $Z$ satisfies the $4s$-RIP for $\delta_{4s}(Z)=\delta$, with probability $\geq 1-2e^{\textcolor{black}{-\gamma(\delta)} m}$, where 
\begin{equation}
    \textcolor{black}{\gamma(\delta)\equiv}\left[c_{0}\left(\frac{\delta}{2}\right)-\frac{4}{b}\left(1+\dfrac{1+\ln(\frac{12}{\delta})}{\ln(\frac{n}{4s})}\right)\right].
    \label{eq:revision2}
\end{equation}
\label{lem:new}    
\end{lemma}

\subsection{Gradient Approximation}
{\tt ZORO} \cite{cai2022zeroth} uses finite differences along random directions to approximate directional derivatives, thought of as linear measurements of $\nabla f(x)$. More precisely, let $\left\{z_{i}\right\}_{i=1}^{m}\subset\mathbb{R}^{n}$ be Rademacher random vectors. For fixed $h > 0$ define the vector $y\in\mathbb{R}^{m}$ with 
\begin{equation}
y_{i}=\dfrac{f(x+h z_{i})-f(x)}{h\sqrt{m}},\quad j=i,\ldots,m.
\label{eq:define_y_vec}
\end{equation}
and the matrix 
\begin{equation}
Z=\dfrac{1}{\sqrt{m}}\left[z_{1}\ldots z_{m}\right]^{T}\in\mathbb{R}^{m\times n}.
\label{eq:define_Z}
\end{equation}

\begin{lemma}
\label{lem:y_error}
Suppose that A1 holds and let $y$ and $Z$ be defined by (\ref{eq:define_y_vec}) and (\ref{eq:define_Z}), respectively. Then
\begin{equation}
y=Z\nabla f(x)+h w,
\label{eq:y_inverse_problem}
\end{equation}
\textcolor{black}{where $w\in\mathbb{R}^{m}$ with} $\|w\|\leq\frac{Ln}{2}$.
\end{lemma}

\begin{proof}
From A1,
\begin{align*}
    \left|f(x+h z_{i})-f(x)-\nabla f(x)^{T}(h z_{i})\right| &\leq \frac{h^2}{2}L\|z_i\|_2^2 \\
    & = \frac{Lh^2}{2}n \quad \text{ as $[z_i]_j = \pm 1$ for all $j$}
\end{align*}
Then
\begin{equation*}
\left|\frac{1}{\sqrt{m}}\dfrac{f(x+h z_{i})-f(x)}{h}-\frac{1}{\sqrt{m}}\nabla f(x)^{T}z_{i}\right|\leq\frac{Ln}{2\sqrt{m}}h.
\end{equation*}
Now, writing $y$ as in \eqref{eq:y_inverse_problem}, it follows from \eqref{eq:define_Z} that
\begin{equation*}
h |\left[w\right]_{i}|=|\left[y\right]_{i}-\left[Z\nabla f(x)\right]_{i}|\leq\frac{Ln}{2\sqrt{m}}h.
\end{equation*}
Therefore, 
\begin{equation*}
|\left[w\right]_{i}|\leq\dfrac{Ln}{2\sqrt{m}},\quad i=1,\ldots,m,
\end{equation*}
which implies that $\|w\|_{2}\leq\frac{Ln}{2}$.\qed
\end{proof}

We view \eqref{eq:y_inverse_problem} as a {\em linear inverse problem} to be solved for $\nabla f(x)$. As in \cite{cai2022zeroth}, given $s\in\mathbb{N}\setminus\left\{0\right\}$, we will compute an approximation to $\nabla f(x)$ by approximately solving the problem
\begin{equation}
 \min_{v\in\mathbb{R}^{n}}  \|Zv-y\| \quad \text{s.t.} \quad                              \|v\|_{0}\leq s,
\label{eq:sparse_recovery}
\end{equation}
using the algorithm CoSaMP proposed in \cite{needell2009cosamp}. 

\begin{lemma}
\label{lem:2.2}
Suppose that A1 and A3 hold. Given $x\in\mathbb{R}^{n}$, a set $\left\{z_{i}\right\}_{i=1}^{m}\subset\mathbb{R}^{n}$ of Rademacher random vectors, $s\in\left\{1,\ldots,n\right\}$ and $h>0$, let $y\in\mathbb{R}^{m}$ and $Z\in\mathbb{R}^{m\times n}$ be defined by (\ref{eq:define_y_vec}) and (\ref{eq:define_Z}), respectively. Denote by $\left\{v^{\ell}\right\}_{\ell\geq 0}$ the sequence generated by applying CoSaMP to the corresponding problem (\ref{eq:sparse_recovery}), with starting point $v^{0}=0$. If $Z$ satisfies the $4s$-RIP for a constant $\delta_{4s}(Z)<0.22665$, then for every $\ell\geq 0$,
\begin{equation}
\|v^{\ell}-\nabla f(x)\|_{2}\leq \left[\left(\frac{13.2}{\left(\frac{2}{p}-1\right)^{\frac{1}{2}}}+\frac{11}{\left(\frac{1}{p}-1\right)}\right)s^{-\frac{2-p}{2p}}+\left(\frac{1}{2}\right)^{\ell}\right]\|\nabla f(x)\|_{2}+\dfrac{11 Ln}{2}h
\label{eq:2.5}
\end{equation} 

\end{lemma}

\begin{proof}
By Lemma \ref{lem:y_error}, 
\begin{equation*}
y=Z\left[\nabla f(x)\right]_{(s)}+e,
\end{equation*}
where 
\begin{equation*}
e=Z\left(\nabla f(x)-\left[\nabla f(x)\right]_{(s)}\right)+hw,
\end{equation*}
with $\|w\|_{2}\leq Ln/2$. Thus, by Theorem 5 and Remark 3 in \cite{foucart2012sparse} we have
\begin{eqnarray*}
\|v^{\ell}-\left[\nabla f(x)\right]_{(s)}\|_{2}&\leq & \left(\frac{1}{2}\right)^{\ell}\|\left[\nabla f(x)\right]_{(s)}\|_{2}+11\|e\|_{2}\\
&\leq & \left(\frac{1}{2}\right)^{\ell}\|\nabla f(x)\|_{2}+11\left\|Z\left(\nabla f(x)-\left[\nabla f(x)\right]_{(s)}\right)\right\|_{2}+11h\|w\|\\
&\leq & \left(\frac{1}{2}\right)^{\ell}\|\nabla f(x)\|_{2}+11\left\|Z\left(\nabla f(x)-\left[\nabla f(x)\right]_{(s)}\right)\right\|_{2}+\frac{11Ln}{2}h.
\end{eqnarray*}
Then, using Lemma \ref{lem:X}, we obtain
\begin{eqnarray}
\|v^{\ell}-\nabla f(x)\|_{2}&\leq &\|v^{\ell}-\left[\nabla f(x)\right]_{(s)}\|_{2}+\|\left[\nabla f(x)\right]_{(s)}-\nabla f(x)\|_{2}\nonumber\\
&\leq & 11\left\|Z\left(\nabla f(x)-\left[\nabla f(x)\right]_{(s)}\right)\right\|_{2}+\frac{11Ln}{2}h\nonumber\\
& & +\left[\left(\frac{1}{2}\right)^{\ell}+\frac{s^{\frac{1}{2}-\frac{1}{p}}}{\left(\frac{2}{p}-1\right)^{\frac{1}{2}}}\right]\|\nabla f(x)\|_{2}
\label{eq:A}
\end{eqnarray}
From Lemmas \ref{lem:X} and \ref{lem:Y}, and the assumption \textcolor{black}{$\delta_{4s}(Z)<0.22665$}, we also have
\small
\begin{eqnarray}
\left\|Z\left(\nabla f(x)-\left[\nabla f(x)\right]_{(s)}\right)\right\|_{2}&\leq & \sqrt{1+\delta_{4s}(Z)}\left(\left\|\nabla f(x)-\left[\nabla f(x)\right]_{(s)}\right\|+\frac{1}{\sqrt{s}}\left\|\nabla f(x)-\left[\nabla f(x)\right]_{(s)}\right\|_{1}\right)\nonumber\\
&\leq &\sqrt{1.22665}\left(\frac{1}{\left(\frac{2}{p}-1\right)^{\frac{1}{2}}}+\frac{1}{\left(\frac{1}{p}-1\right)}\right)s^{\frac{1}{2}-\frac{1}{p}}\|\nabla f(x)\|_{2}
\label{eq:B}
\end{eqnarray}
\normalsize
Finally, combining (\ref{eq:A}) and (\ref{eq:B}), we conclude that
\small
\begin{equation*}
\|v^{\ell}-\nabla f(x)\|_{2}\leq\left[\left(\frac{11\sqrt{1.22665}+1}{\left(\frac{2}{p}-1\right)^{\frac{1}{2}}}+\frac{11}{\left(\frac{1}{p}-1\right)}\right)s^{-\frac{2-p}{2p}}+\left(\frac{1}{2}\right)^{\ell}\right]\|\nabla f(x)\|_{2}+\frac{11Ln}{2}h,
\end{equation*}
\normalsize
which implies that (\ref{eq:2.5}) is true.\qed

\end{proof}

\noindent The following corollary may be deduced from Lemma~\ref{lem:2.2} by direct computation. 
\begin{corollary}
\label{cor:2.1}
Let $\theta,\epsilon\in (0,1)$. Under the assumptions of Lemma \ref{lem:2.2}, if 
\begin{equation}
\textcolor{black}{\ell\geq}\left\lceil\dfrac{\log(\theta/4)}{\log(0.5)}\right\rceil,
\label{eq:2.7}
\end{equation}
\begin{equation}
s\geq s(\theta,p):=\left[\frac{4}{\theta}\left(\frac{13.2}{\left(\frac{2}{p}-1\right)^{\frac{1}{2}}}+\frac{11}{\left(\frac{1}{p}-1\right)}\right)\right]^{\frac{2p}{2-p}} ,
\label{eq:effective_sparsity}
\end{equation}
and
\begin{equation}
0<h\leq\dfrac{\theta}{11 L n}\epsilon,
\label{eq:2.9}
\end{equation}
then, for $g=v^{\ell}$ we have
\begin{equation}
\|g-\nabla f(x)\|\leq\dfrac{\theta}{2}\|\nabla f(x)\|+\dfrac{\theta}{2}\epsilon.
\label{eq:2.10}
\end{equation}
We call the quantity $s(\theta,p)$ defined in \eqref{eq:effective_sparsity} the effective sparsity level. We emphasize that $s(\theta,p)$ is independent of $n$. 
\end{corollary}

\noindent \textcolor{black}{Note that, from Lemma 5, the number of iterations of CoSaMP only affects the factor \( (1/2)^{\ell} \) in (\ref{eq:2.5}), leaving the other factors unchanged. The per-iteration cost of CoSaMP is \( \mathcal{O}(mn + ms^2) \) operations. Thus, Corollary 1 suggests that it is sufficient to perform 
\begin{equation*}
\ell=\left\lceil\dfrac{\log(\theta/4)}{\log(0.5)}\right\rceil
\end{equation*}
iterations to guarantee (\ref{eq:2.10}), provided that (\ref{eq:effective_sparsity}) and (\ref{eq:2.9}) hold, thereby saving computational effort. For example, if $\theta=0.25$, it follows that $\ell=4$ iterations of CoSaMP would be enough.}

\subsection{Sufficient Decrease Condition}

From Corollary \ref{cor:2.1}, we see that if $s$ is sufficiently large and $h$ is sufficiently small, then CoSaMP applied to (\ref{eq:sparse_recovery}) is able to find a vector $g$ that satisfies 
\begin{equation}
\|g-\nabla f(x)\|\leq\dfrac{\theta}{2}\|\nabla f(x)\|+\dfrac{\theta}{2}\epsilon
\label{eq:2.11}
\end{equation}
whenever the matrix $Z$ satisfies the $4s$-RIP with parameter $\delta_{4s}(Z)<0.22665$. The next lemma justifies our interest in this type of gradient approximation.

\begin{lemma}[Descent Lemma]
\label{lem:2.3}
Suppose that A1 holds. Given $\epsilon\in (0,1)$ and $\theta\in [0,1/2)$, let $x\in\mathbb{R}^{n}$ and $g\in\mathbb{R}^{n}$ be such that $\|\nabla f(x)\|>\epsilon$ and (\ref{eq:2.11}) hold. Moreover, let 
\begin{equation*}
x^{+}=x-\frac{1}{\sigma}g
\end{equation*}
for some $\sigma>0$. If 
\begin{equation}
\sigma\geq\dfrac{(\theta +1)^{2}}{(1-2\theta)}L
\label{eq:2.12}
\end{equation}
then
\begin{equation}
f(x)-f(x^{+})\geq\dfrac{1}{2\sigma}\epsilon^{2}.
\label{eq:2.13}
\end{equation}
\end{lemma}

\begin{proof}
Since $\|\nabla f(x)\|>\epsilon$, it follows from (\ref{eq:2.11}) that
\begin{equation}
\|\nabla f(x)-g\|\leq\theta\|\nabla f(x)\|.
\label{eq:2.14}
\end{equation}
Using assumption A1, the Cauchy-Schwarz inequality and (\ref{eq:2.14}) we get
\begin{eqnarray}
f(x^{+})&\leq & f(x)+\nabla f(x)^{T}\left(-\frac{1}{\sigma}g\right)+\dfrac{L}{2}\|\frac{1}{\sigma}g\|^{2}\nonumber\\
            & =   & f(x)-\frac{1}{\sigma}\nabla f(x)^{T}g+\dfrac{L}{2\sigma^{2}}\|g\|^{2}\nonumber\\
            &  =  & f(x)-\frac{1}{\sigma}\nabla f(x)^{T}\left[\nabla f(x)-\nabla f(x)+g\right]+\dfrac{L}{2\sigma^{2}}\|g\|^{2}\nonumber\\
            &\leq & f(x)-\frac{1}{\sigma}\|\nabla f(x)\|^{2}+\frac{1}{\sigma}\|\nabla f(x)\|\|\nabla f(x)-g\|+\dfrac{L}{2\sigma^{2}}\|g\|^{2}\nonumber\\
            &\leq & f(x)-\frac{1}{\sigma}\|\nabla f(x)\|^{2}+\dfrac{\theta}{\sigma}\|\nabla f(x)\|^{2}+\dfrac{L}{2\sigma^{2}}\|g\|^{2}.
\label{eq:2.15}
\end{eqnarray}
On the other hand, by (\ref{eq:2.11}) we also have
\begin{equation*}
\|g\|\leq \|g-\nabla f(x)\|+\|\nabla f(x)\|\leq\theta\|\nabla f(x)\|+\|\nabla f(x)\|=(\theta+1)\|\nabla f(x)\|,
\end{equation*}
which gives
\begin{equation}
\|g\|^{2}\leq (\theta +1)^{2}\|\nabla f(x)\|^{2}.
\label{eq:2.16}
\end{equation}
Combining (\ref{eq:2.15}) and (\ref{eq:2.16}) it follows that
\begin{eqnarray*}
f(x^{+})&\leq & f(x)-\left(\frac{1}{\sigma}-\frac{\theta}{\sigma}-\frac{L}{2\sigma^{2}}(\theta+1)^{2}\right)\|\nabla f(x)\|^{2}\\
            & =   &f(x)-\left(1-\theta-\frac{L}{2\sigma}(\theta+1)^{2}\right)\frac{1}{\sigma}\|\nabla f(x)\|^{2},
\end{eqnarray*}
and so
\begin{equation}
f(x)-f(x^{+})\geq \left(1-\theta-\frac{L}{2\sigma}(\theta+1)^{2}\right)\frac{1}{\sigma}\|\nabla f(x)\|^{2}.
\label{eq:2.17}
\end{equation}
By (\ref{eq:2.12}) we have
\begin{equation}
1-\theta-\frac{L}{2\sigma}(\theta+1)^{2}\geq\dfrac{1}{2}.
\label{eq:2.18}
\end{equation}
Finally, combining (\ref{eq:2.17}), (\ref{eq:2.18}) and using the inequality $\|\nabla f(x)\|>\epsilon$, we obtain (\ref{eq:2.13}).\qed
\end{proof}

The lemma below suggests a way to construct a gradient approximation $g$ that satisfies (\ref{eq:2.11}) with probability one. 

\begin{lemma}
Suppose that A1 holds. Given $x\in\mathbb{R}^{n}$, $\epsilon,\textcolor{black}{\theta}>0$, let $g\in\mathbb{R}^{n}$ be given by
\begin{equation}
g_{i}=\dfrac{f(x+h e_{i})-f(x)}{h},\quad i=1,\ldots,n,
\label{eq:extra1}
\end{equation}
with 
\begin{equation}
0<h\leq\dfrac{2\theta}{L\sqrt{n}}\epsilon.
\label{eq:extra2}
\end{equation}
If $\|\nabla f(x)\|>\epsilon$, then $g$ satisfies (\ref{eq:2.11}).
\label{lem:2.4}
\end{lemma}

\begin{proof}
By (\ref{eq:extra1}), A1, (\ref{eq:extra2}) and $\epsilon<\|\nabla f(x)\|$, we have
\begin{equation*}
\|\nabla f(x)-g\|\leq\dfrac{L\sqrt{n}}{2}h\leq\theta\epsilon<\frac{\theta}{2}\|\nabla f(x)\|+\frac{\theta}{2}\epsilon.
\end{equation*}
\qed
\end{proof}

\section{New Method and its Complexity}

The proposed method is presented below as Algorithm 1. We call it Fully Adaptive {\tt ZORO}, or {\tt ZORO-FA}, as the target sparsity ($s$), sampling radius ($h$), step-size ($1/\sigma$), and number of queries ($m$) are all adaptively selected. {\tt ZORO-FA} includes a safeguard mechanism, ensuring that it will make progress even when $\nabla f(x)$ {\em does not} satisfy the compressibility assumption A3. 

\begin{mdframed}
\noindent\textbf{Algorithm 1.} Fully Adaptive {\tt ZORO}, or {\tt ZORO-FA}. 
\\[0.2cm]
\textbf{Step 0.} Given $x_{0}\in\mathbb{R}^{n}$, $\epsilon\in (0,1)$, $\theta\in (0,1/2)$, $b \geq 1$, and $\sigma_{0}>0$. Choose \textcolor{black}{$s_{0}\in\mathbb{N}$} such that $\lceil b s_{0}\ln(n)\rceil\textcolor{black}{\leq n/4}$, sample i.i.d. Rademacher random vectors $\{z_i\}_{i=1}^n$, and set $k:=0$.
\\[0.2cm]
\textbf{Step 1.} Set $j:=0$.
\\[0.2cm]
\textbf{Step 2.1} Set $s_j = 2^js_0$, $\sigma_j = 2^j\sigma_0$, and $m_{j}=\left\lceil bs_{j}\ln(n)\right\rceil$. If $m_{j}\geq n$ go to Step 2.4.
\\[0.2cm]
\noindent\textbf{Step 2.2} Set
\begin{equation*}
    h_{j}=\dfrac{\theta}{11 n \sigma_{j}}\epsilon\quad\text{and}\quad Z^{(j)} = \frac{1}{\sqrt{m_j}} \begin{bmatrix} z_1 & \ldots & z_{m_j} \end{bmatrix}.
\end{equation*}
Compute $y_k^{(j)}\in\mathbb{R}^{m_{j}}$ by
\begin{equation*}
\textcolor{black}{\left[y^{(j)}_{k}\right]_{i}=\dfrac{f(x_{k}+h_{j}z_{i})-f(x_{k})}{\sqrt{m_j}h_{j}},\quad i=1,\ldots,m_{j}}.
\end{equation*}
\\[0.2cm]
\noindent\textbf{Step 2.3} Compute
\begin{equation*}
    g_{k}^{(j)}\approx \arg\min_{\|g\|_{0}\leq s_{j}}\|Z^{(j)}g-y^{(j)}_{k}\|_{2}
\end{equation*}
by applying $\left\lceil\dfrac{\log\left(\theta/4\right)}{\log(0.5)}\right\rceil$ iterations of CoSaMP. Go to Step 2.5.

\noindent\textbf{Step 2.4} For 
\begin{equation*}
    h_j = \frac{2\theta}{\sigma_j\sqrt{n}}\epsilon,
\end{equation*}
compute $g_k^{(j)} \in \mathbb{R}^n$ by 
\begin{equation*}
    [g_k^{(j)}]_{\ell} = \frac{f(x_k + h_je_{\ell}) - f(x_k)}{h_j}, \ \ell = 1,\ldots, n.
\end{equation*}
\\[0.2cm]
\noindent\textbf{Step 2.5} Let $x_{k,j}^{+}=x_{k}-\left(\frac{1}{\sigma_{j}}\right)g_{k}^{(j)}$. If 
\begin{equation}
    f(x_{k})-f(x_{k,j}^{+})\geq\dfrac{1}{2\sigma_{j}}\epsilon^{2}
    \label{eq:decrease}
\end{equation}
holds, set $j_{k}=j$, $g_{k}=g_{k}^{(j)}$, $s_{k}=s_{j_{k}}$, $\sigma_{k}=\sigma_{j_{k}}$ and go to Step 3. Otherwise, set $j:=j+1$ and go to Step 2.1.
\\[0.2cm]
\noindent\textbf{Step 3.} Define $x_{k+1}=x_{k,j_{k}}^{+}$ set $k:=k+1$ and go to Step 1.
\end{mdframed}
\textcolor{black}{
\begin{remark}
    Note that vectors $\left\{z_{i}\right\}_{i=1}^{n}$ are sampled only once, at Step 0. Moreover, in step 2.2 it is important to evaluate the finite difference using the {\em unscaled} $z_i$ (i.e. compute $f(x_k + hz_i)$), not the $i$-th column of $Z^{(j)}$.
\end{remark}
}
\begin{remark}
Note that {\tt ZORO-FA} bears some resemblence to the adaptive {\tt ZORO}, or {\tt adaZORO} algorithm presented in \cite{cai2022zeroth}. However, that algorithm only selects the target sparsity ($s$) adaptively. Moreover, we are able to provide theoretical complexity guarantees for {\tt ZORO-FA}, whereas no such guarantees are presented for {\tt adaZORO}. 
\end{remark}

\subsection{Oracle Complexity Without Compressible Gradients}

\textcolor{black}{In this subsection, we analyze Algorithm 1 under Assumptions A1 and A2 only; that is, we do not assume that the gradients of \( f(\,\cdot\,) \) are compressible. We start by establishing an upper bound on the number of inner loops \( j_k \) required to find a point that yields (\ref{eq:decrease}).
}

\begin{lemma}
\label{lem:plus1}
Suppose that A1 holds. If $\|\nabla f(x_{k})\|>\epsilon$, then
\begin{equation}
0\leq j_{k}<\max\left\{1,\log_{2}\left(\dfrac{2n}{bs_{0}\ln(n)}\right),\log_{2}\left(\dfrac{2(\theta+1)^{2}}{(1-2\theta)}\dfrac{L}{\sigma_{0}}\right)\right\}.
\label{eq:plus1}
\end{equation}
\end{lemma}

\begin{proof}
From Steps 1 and 2 of Algorithm 1, we see that $j_{k}$ is the smallest non-negative integer for which (\ref{eq:decrease}) holds. \textcolor{black}{Suppose by contradiction that (\ref{eq:plus1}) is not true}, i.e.,
\begin{equation}
j_{k}\geq \max\left\{1,\log_{2}\left(\dfrac{2n}{bs_{0}\ln(n)}\right),\log_{2}\left(\dfrac{2(\theta+1)^{2}}{(1-2\theta)}\dfrac{L}{\sigma_{0}}\right)\right\}.
\label{eq:plus2}
\end{equation}
Then
\begin{equation*}
m_{j_{k}-1}=\lceil{b s_{j_{k}-1}\ln(n)\rceil}\geq b\left(2^{j_{k}-1}s_{0}\right)\ln(n)\geq b\left(\dfrac{n}{b\ln(n)}\right)\ln(n)=n.
\end{equation*}
Consequently, by Steps 2.1 and 2.4 of Algorithm 1, it follows that $g_{k}^{(j_{k}-1)}$ was defined by
\begin{equation}
\left[g_{k}^{(j_{k}-1)}\right]_{i}=\dfrac{f(x_{k}+h_{j_{k}-1}e_{i})-f(x_{k})}{h_{j_{k}-1}},\quad i=1,\ldots,n,
\label{eq:plus3}
\end{equation}
with
\begin{equation}
h_{j_{k}-1}=\dfrac{2\theta}{\sigma_{j_{k}-1}\sqrt{n}}\epsilon.
\label{eq:plus4}
\end{equation}
In addition, (\ref{eq:plus2}) also implies that
\begin{equation}
\sigma_{j_{k}-1}=2^{j_{k}-1}\sigma_{0}\geq\dfrac{(\theta+1)^{2}}{(1-2\theta)}L
\label{eq:plus5}
\end{equation}
Since $\theta\in (0,1)$, (\ref{eq:plus4}) and (\ref{eq:plus5}) imply that
\begin{equation}
0<h_{j_{k}-1}<\dfrac{2\theta}{L\sqrt{n}}\epsilon.
\label{eq:plus6}
\end{equation}
In view of (\ref{eq:plus3}), (\ref{eq:plus6}) and assumption $\|\nabla f(x_{k})\|>\epsilon$, it follows from Lemma \ref{lem:2.4} that (\ref{eq:2.11}) holds for $x=x_{k}$ and $g=g_{k}^{(j_{k}-1)}$. Thus, by (\ref{eq:plus5}) and Lemma \ref{lem:2.3} we would have 
\begin{equation*}
f(x_{k})-f(x_{k,j_{k}-1}^{+})\geq\dfrac{1}{2\sigma_{j_{k}-1}}\epsilon^{2},
\end{equation*}
contradicting the definition of $j_{k}$.\qed
\end{proof}
As a consequence of Lemma \ref{lem:plus1}, we have the following result.

\begin{lemma}
Suppose that A1 holds. If $\|\nabla f(x_{k})\|>\epsilon$ then
\begin{equation*}
s_{k}\leq\max\left\{2s_{0},\left[\dfrac{2n}{b\ln(n)}\right],\left[\dfrac{2(\theta+1)^{2}L}{(1-2\theta)\sigma_{0}}\right]s_{0}\right\}.
\end{equation*}
and
\begin{equation*}
\sigma_{k}\leq\max\left\{2\sigma_{0},\left[\dfrac{2n}{b s_{0}\ln(n)}\right]\sigma_{0},\left[\dfrac{2(\theta+1)^{2}}{(1-2\theta)}\right]L\right\}.
\end{equation*}
\label{lem:plus2}
\end{lemma}
Denote
\begin{equation}
T(\epsilon)=\inf\left\{k\in\mathbb{N}\,:\,\|\nabla f(x_{k})\|\leq\epsilon\right\}.
\label{eq:hitting}
\end{equation}

\begin{theorem}
\label{thm:plus1}
Suppose that A1-A2 hold and let $\left\{x_{k}\right\}_{k=0}^{T(\epsilon)}$ be generated by Algorithm 1. Then
\begin{equation}
T(\epsilon)\leq 2\max\left\{2\sigma_{0},\left[\dfrac{2n}{b s_{0}\ln(n)}\right]\sigma_{0},\left[\dfrac{2(\theta+1)^{2}}{(1-2\theta)}\right]L\right\}(f(x_{0})-f_{low})\epsilon^{-2}.
\label{eq:plus7}
\end{equation}
\end{theorem}

\begin{proof}
From Step 2.5 of Algorithm 1, we have
\begin{equation*}
f(x_{k})-f(x_{k+1})\geq\dfrac{1}{2\sigma_{k}}\epsilon^{2},\quad\text{for}\,\,k=0,\ldots,T(\epsilon)-1.
\end{equation*}
Summing up these inequalities and using A2 and Lemma \ref{lem:plus2}, we obtain
\begin{eqnarray*}
f(x_{0})-f_{low}&\geq &\sum_{k=0}^{T(\epsilon)-1} f(x_{k})-f(x_{k+1})\geq \sum_{k=0}^{T(\epsilon)-1}\dfrac{1}{2\sigma_{k}}\epsilon^{2}\\
                &\geq &\dfrac{T(\epsilon)}{2\max\left\{2\sigma_{0},\left[\dfrac{2n}{b s_{0}\ln(n)}\right]\sigma_{0},\left[\dfrac{2(\theta+1)^{2}}{(1-2\theta)}\right]L\right\}}\epsilon^{2},
\end{eqnarray*}
which implies that (\ref{eq:plus7}) is true.\qed
\end{proof}

\begin{corollary}
Suppose that A1-A2 hold and denote by $FE_{T(\epsilon)}$ the total number of function evaluations performed by Algorithm 1 up to the $(T(\epsilon)-1)$-st iteration. Then
\small
\begin{eqnarray}
FE_{T(\epsilon)}&\leq & 2(n+1)\max\left\{1,\log_{2}\left(\dfrac{2n}{b s_{0}\ln(n)}\right),\log_{2}\left(\dfrac{2(\theta+1)^{2}}{1-2\theta}\dfrac{L}{\sigma_{0}}\right)\right\}\times\nonumber\\
& & \left[2\max\left\{2\sigma_{0},\left[\dfrac{2n}{b s_{0}\ln(n)}\right]\sigma_{0},\left[\dfrac{2(\theta+1)^{2}}{(1-2\theta)}\right]L\right\}(f(x_{0})-f_{low})\epsilon^{-2}\right].
\label{eq:plus8}
\end{eqnarray}
\label{cor:plus1}
\end{corollary}
\normalsize

\begin{proof}
Let us consider the $k$th iteration of Algorithm 1. For any $j\in\left\{0,\ldots,j_{k}\right\}$, the computation of $g_{k}^{(j)}$ requires at most $n+1$ function evaluations. Moreover, checking the acceptance condition (\ref{eq:decrease}) at Step 2.5 requires one additional function evaluation. Thus, the total number of function evaluations at the $k$th iteration is bounded from above by $(n+1)(j_{k}+1)$. Therefore, by Lemma \ref{lem:plus1},
\small
\begin{eqnarray}
FE_{T(\epsilon)}&\leq & (n+1)\sum_{k=0}^{T(\epsilon)-1}(j_{k}+1)\leq 2(n+1)\sum_{k=0}^{T(\epsilon)-1}j_{k}\nonumber\\
                &\leq & 2(n+1)\max\left\{1,\log_{2}\left(\dfrac{2n}{bs_{0}\ln(n)}\right),\log_{2}\left(\dfrac{2(\theta+1)^{2}}{(1-2\theta)}\dfrac{L}{\sigma_{0}}\right)\right\}T(\epsilon).
\label{eq:plus9}
\end{eqnarray}
\normalsize
Finally, combining (\ref{eq:plus9}) and Theorem \ref{thm:plus1} we conclude that (\ref{eq:plus8}) is true.\qed
\end{proof}

\begin{remark}
Recall that \textcolor{black}{$m_{0}=\left\lceil b s_{0}\log(n)\right\rceil$}. Thus, it follows from Corollary \ref{cor:plus1} that 
\begin{equation*}
    FE_{T(\epsilon)}\leq\mathcal{O}\left(n\left[\left(\dfrac{2n}{m_{0}}\right)\log_{2}\left(\dfrac{2n}{m_{0}}\right)\right]\epsilon^{-2}\right).
\end{equation*}
\textcolor{black}{Therefore, our complexity bound ranges from \( \mathcal{O}(n\epsilon^{-2}) \) to \( \mathcal{O}(n^{2}\epsilon^{-2}) \), depending on the choice of \( s_0 \). When \( s_0 \) is chosen such that \( m_0 = \left\lceil b s_0 \ln(n) \right\rceil = n/4 \)}, Algorithm~1 requires no more than $\mathcal{O}\left(n\epsilon^{-2}\right)$ function evaluations to find an $\epsilon$-approximate stationary point of $f(\,\cdot\,)$.  \textcolor{black}{Regarding the dependence on $n$ and $\epsilon$, this bound is consistent, in order, with the complexity bounds established for the method DFQRM in  \cite{grapiglia2024quadratic}}. However, to take advantage of potential compressibility of gradients, \textcolor{black}{one should take $s_{0}$ sufficiently small such that $m_{0}\ll n$}. In this case, the worst-case complexity bound becomes $\mathcal{O}\left(n^{2}\epsilon^{-2}\right)$.
\end{remark}

\subsection{Oracle Complexity Assuming Compressible Gradients}

\textcolor{black}{In this subsection, we analyze Algorithm 1 under the additional assumption that the gradients of \( f(\,\cdot\,) \) are compressible (Assumption A3).} Note that $$J\equiv\left\lceil\log_{2}\left(\frac{n}{b s_{0}\ln(n)}\right)\right\rceil$$ 
is the smallest \textcolor{black}{positive} integer such that $m_{J}=\lceil b s_{J}\ln(n)\rceil\geq n$. Therefore,
\begin{equation*}
    Z^{(j)}=\dfrac{1}{\sqrt{m_{j}}}\left[z_{1}\ldots z_{m_{j}}\right],\quad j=0,\ldots,J-1
\end{equation*}
are the only possible sensing matrices used in Algorithm 1. \textcolor{black}{Moreover, since $s_{0}$ is chosen such that $\left\lceil bs_{0}\ln(n)\right\rceil\leq n/4$, we have $J\geq 2$.} In this subsection we will consider the following additional assumptions:
\begin{itemize}
\item[\textbf{A4.}] Under A1 and A3, the index
\begin{equation}
    j^{*}:=\left\lceil \max\left\{1,\log_{2}\left(\dfrac{s(\theta,p)}{s_{0}}\right),\log_{2}\left(\dfrac{(\theta+1)^{2}}{(1-2\theta)}\dfrac{L}{\sigma_{0}}\right)\right\} \right\rceil,
\label{eq:magic}
\end{equation}
with $s(\theta,p)$ defined in (\ref{eq:effective_sparsity}), belongs to the interval $[1,J-1)$.
\item[]
\item[\textbf{A5.}] For $j^{*}$ defined in (\ref{eq:magic}), the corresponding matrix $Z^{(j^{*})}$ satisfies the $4s_{j^{*}}$-RIP with parameter $\delta_{4s_{j^{*}}}(Z^{(j^{*})})<0.22665$.
\end{itemize}

\begin{remark}
\textcolor{black}{Assumption~4 implicitly requires \( s(\theta, p) \) to be sufficiently small. In view of~\eqref{eq:effective_sparsity}, this means that the gradients should be \( p \)-compressible (Assumption A3), with \( p \) sufficiently close to zero. Moreover, by Lemma 3, Assumption A5 holds with high probability when \( b \) in Step 0 is chosen sufficiently large. This fact is clarified in Corollary~4.
}
\end{remark}

\begin{lemma}
\label{lem:mais1}
Suppose that A1, A4 and A5 hold. If $\|\nabla f(x_{k})\|>\epsilon$, then
\begin{equation}
0\leq j_{k}\leq \left\lceil \max\left\{1,\log_{2}\left(\dfrac{s(\theta,p)}{s_{0}}\right),\log_{2}\left(\dfrac{(\theta+1)^{2}}{(1-2\theta)}\dfrac{L}{\sigma_{0}}\right)\right\} \right\rceil.
\label{eq:mais1}
\end{equation}
\end{lemma}

\begin{proof}
Let us show that, for $j^{*}$ defined in (\ref{eq:magic}), the corresponding point $x_{k,j^{*}}^{+}$ satisfies (\ref{eq:decrease}). Indeed, by A4 and the definition of $J$, it follows that $g_{k}^{(j^{*})}$ was obtained at Step 2.3 of Algorithm 1, using the sensing matrix $Z^{(j^{*})}$, and the vector $y_{k}^{(j^{*})}$ obtained with
\begin{equation}
h_{j^{*}}=\dfrac{\theta}{11 n\sigma_{j^{*}}}.
 \label{eq:mais3}   
\end{equation}
In view of (\ref{eq:magic}) and Assumption A5, the matrix $Z^{(j^{*})}$ satisfies the $4s_{j^{*}}$-RIP with 
\begin{equation}
s_{j^{*}}=2^{j^{*}}s_{0}\geq s(\theta,p).
\label{eq:mais4}
\end{equation}
Moreover, (\ref{eq:magic}) also implies that 
\begin{equation}
\sigma_{j^{*}}=2^{j^{*}}\sigma_{0}\geq\dfrac{(\theta+1)^{2}}{(1-2\theta)}L.
\label{eq:mais5}
\end{equation}
Combining (\ref{eq:mais3}) and (\ref{eq:mais5}), we have that
\begin{equation}
h_{j^{*}}<\dfrac{\theta}{11 n L}\epsilon.
\label{eq:mais6}
\end{equation}
Then, it follows from (\ref{eq:mais4}), (\ref{eq:mais6}) and Corollary \ref{cor:2.1} that (\ref{eq:2.11}) holds for $x=x_{k}$ and $g=g_{k}^{(j^{*})}$. Thus, by (\ref{eq:mais5}) and Lemma \ref{lem:2.3} we would have 
\begin{equation*}
f(x_{k})-f(x_{k,j^{*}}^{+})\geq\dfrac{1}{2\sigma_{j^{*}}}\epsilon^{2}.
\end{equation*}
Since $j_{k}$ is the smallest non-negative integer $j$ for which (\ref{eq:decrease}) holds, we conclude that $j_{k}\leq j^{*}$, and so (\ref{eq:mais1}) is true.\qed
\end{proof}
Lemma \ref{lem:mais1} implies the following upper bounds on $s_{k}$ and $\sigma_{k}$.
\begin{lemma}
Suppose that A1, A4 and A5 hold. If $\|\nabla f(x_{k})\|>\epsilon$ then
\begin{equation*}
s_{k}\leq\max\left\{2s_{0},2s(\theta,p),\left[\dfrac{2(\theta+1)^{2}L}{(1-2\theta)\sigma_{0}}\right]s_{0}\right\}.
\end{equation*}
and
\begin{equation*}
\sigma_{k}\leq\max\left\{2\sigma_{0},\left[\dfrac{2 s(\theta,p)}{s_{0}}\right]\sigma_{0},\left[\dfrac{2(\theta+1)^{2}}{(1-2\theta)}\right]L\right\}.
\end{equation*}
\label{lem:mais2}
\end{lemma}

\begin{theorem}
\label{thm:mais1}
Suppose that A1-A5 hold and let $\left\{x_{k}\right\}_{k=0}^{T(\epsilon)}$ be generated by Algorithm 1, where $T(\epsilon)$ is defined by (\ref{eq:hitting}). Then
\begin{equation}
T(\epsilon)\leq 2\max\left\{2\sigma_{0},\left[\dfrac{2 s(\theta,p)}{ s_{0}}\right]\sigma_{0},\left[\dfrac{2(\theta+1)^{2}}{(1-2\theta)}\right]L\right\}(f(x_{0})-f_{low})\epsilon^{-2}.
\label{eq:mais7}
\end{equation}
\end{theorem}

\begin{proof}
As in the proof of Theorem \ref{thm:plus1}, we have
\begin{equation}
f(x_{0})-f_{low}\geq\dfrac{T(\epsilon)}{2\max_{k=0,\ldots,T(\epsilon)-1}\left\{\sigma_{k}\right\}}
\label{eq:mais8}
\end{equation}
Thus, combining (\ref{eq:mais8}) and the second inequality in Lemma \ref{lem:mais2} we see that (\ref{eq:mais7}) is true.\qed
\end{proof}

\begin{corollary}
Suppose that A1-A5 hold and denote by $FE_{T(\epsilon)}$ the total number of function evaluations performed by Algorithm 1 up to the $(T(\epsilon)-1)$-st iteration. Then
\small
\begin{eqnarray}
FE_{T(\epsilon)}&\leq & 4b\left[1+\max\left\{1,\log_{2}\left(\dfrac{s(\theta,p)}{s_{0}}\right),\log_{2}\left(\dfrac{(\theta+1)^{2}}{(1-2\theta)}\dfrac{L}{\sigma_{0}}\right)\right\}\right]\times\nonumber\\
& & \max\left\{2s_{0},2s(\theta,p),\left[\dfrac{2(\theta+1)^{2}L}{(1-2\theta)\sigma_{0}}\right]s_{0}\right\}\ln(n)\times\nonumber\\
& & \max\left\{2\sigma_{0},\left[\dfrac{2 s(\theta,p)}{ s_{0}}\right]\sigma_{0},\left[\dfrac{2(\theta+1)^{2}}{(1-2\theta)}\right]L\right\}(f(x_{0})-f_{low})\epsilon^{-2}.
\label{eq:mais9}
\end{eqnarray}
\label{cor:mais1}
\end{corollary}
\normalsize

\begin{proof}
Let us consider the $k$th iteration of Algorithm 1. By Lemma \ref{lem:mais1} and Assumption A4, we have $m_{j}=b s_{j}\ln(n)<n$ for $j=0,\ldots,j_{k}$. Then, for $j\in\left\{0,\ldots,j_{k}\right\}$, the computation of $g_{k}^{(j)}$ requires at most $(m_{j}+1)$ function evaluations. Moreover, checking the acceptance condition (\ref{eq:decrease}) requires one additional function evaluation. Thus, the total number of function evaluations at the $k$th iteration is bounded from above by
\begin{eqnarray*}
\sum_{j=0}^{j_{k}}(m_{j}+2)&\leq &2\sum_{j=0}^{j_{k}}b s_{j}\ln(n)\leq 2b(1+j_{k}) s_{j_{k}}\ln(n)\\
                     &\leq &2b\left[1+\left\lceil\max\left\{1,\log_{2}\left(\dfrac{s(\theta,p)}{s_{0}}\right),\log_{2}\left(\dfrac{(\theta+1)^{2}}{(1-2\theta)}\dfrac{L}{\sigma_{0}}\right)\right\}\right\rceil\right]\times\nonumber\\
& & \max\left\{2s_{0},2s(\theta,p),\left[\dfrac{2(\theta+1)^{2}L}{(1-2\theta)\sigma_{0}}\right]s_{0}\right\}\ln(n).
\end{eqnarray*}
Consequently
\begin{eqnarray}
FE_{T(\epsilon)}&\leq & 2b\left[1+\left\lceil\max\left\{1,\log_{2}\left(\dfrac{ s(\theta,p)}{s_{0}}\right),\log_{2}\left(\dfrac{(\theta+1)^{2}}{(1-2\theta)}\dfrac{L}{\sigma_{0}}\right)\right\}\right\rceil\right]\times\nonumber\\
& & \max\left\{2s_{0},2s(\theta,p),\left[\dfrac{2(\theta+1)^{2}L}{(1-2\theta)\sigma_{0}}\right]s_{0}\right\}\ln(n)T(\epsilon).
\label{eq:mais10}
\end{eqnarray}
\normalsize
Finally, combining (\ref{eq:mais10}) and Theorem \ref{thm:mais1} we conclude that (\ref{eq:mais9}) is true.\qed
\end{proof}

\noindent When $s_{0}\leq s(\theta,p)$, it follows from Corollary \ref{cor:mais1} that 
\begin{equation*}
FE_{T(\epsilon)}\leq\mathcal{O}\left(\left[\dfrac{s(\theta,p)}{s_{0}}\right]\log_{2}\left(\dfrac{s(\theta,p)}{s_{0}}\right)s(\theta,p)\ln(n)\epsilon^{-2}\right).
\end{equation*}
Assuming further that $s_0 = \mathcal{O}(s(\theta,p))$, i.e., the initial sparsity $s_0$ is not too far from $s(\theta,p)$, this bound reduces to 
\begin{equation}
FE_{T(\epsilon)}\leq\mathcal{O}\left(s(\theta,p)\ln(n)\epsilon^{-2}\right).
\label{eq:revision3}
\end{equation}
\textcolor{black}{
\begin{corollary}
Suppose Assumptions~A1--A4 hold, and consider the functions \( c_1(\cdot) \) and \( \gamma(\cdot) \) defined in~\eqref{eq:revision1} and~\eqref{eq:revision2}, respectively. If \( b \geq c_1(0.22664) \), then, with probability at least
\begin{equation}
1 - 2 e^{-\gamma(0.22664) m_{j^{*}}},
\label{eq:revision4}
\end{equation}
Algorithm~1 requires no more than \(\mathcal{O}\bigl(s(\theta,p) \ln(n) \epsilon^{-2}\bigr)\) function evaluations to find an \(\epsilon\)-approximate stationary point of \( f(\cdot) \).
\end{corollary}
}

\begin{proof}
\textcolor{black}{If Assumption~A5 also holds, then the stated complexity bound follows from Corollary~\ref{cor:mais1}. Since \( b \geq c_1(0.22664) \), it follows from Lemma~\ref{lem:new} that, with probability at least~\eqref{eq:revision4}, the sensing matrix \( Z^{(j^{*})} \) satisfies the \( 4s_{j^{*}} \)-RIP with \( \delta_{4s_{j^{*}}}(Z^{(j^{*})}) < 0.22665 \). This implies that Assumption~A5 holds with probability at least (\ref{eq:revision4}). Therefore, the stated complexity bound holds with the same probability.}\qed
\end{proof}

\section{Illustrative Numerical Experiments}
In the following experiments, we compare {\tt ZORO-FA} to a representative sample of {\em deterministic} DFO algorithms, namely {\tt Nelder-Mead} \cite{NM}, {\tt DFQRM} \cite{grapiglia2024quadratic}, \textcolor{black}{the zeroth-order variant \cite{kozak2023} of stochastic subspace descent \cite{Kozak} (henceforth: {\tt SSD}), direct search in reduces spaces \cite{Roberts} (henceforth: {\tt DS-RS}),} and {\tt ZORO} \textcolor{black}{and {\tt adaZORO}} \cite{cai2022zeroth}. We avoid methods which incorporate curvature, or ``second-order'' information such as {\tt CARS} \cite{kim2021curvature} or the derivative-free variation of {\tt L-BFGS} studied in \cite{shi2023numerical}, noting that such functionality could be added to {\tt ZORO-FA} fairly easily. \textcolor{black}{We also do not compare against model-based approaches which use a quadratic model, such as {\tt DFBGN} \cite{Cartis} or  {\tt NEWUOA} \cite{powell2006newuoa,Zhang_2023}, as these offer different trade-offs between run-time, query complexity, and solution quality.}

\subsection{Sparse Gradient Functions}
\label{sec:Sparse_Benchmarking}
We test {\tt ZORO-FA} on \textcolor{black}{two} challenging test functions known to exhibit gradient sparsity, namely the max-s-squared function ($f_{ms}$) \cite{cai2022one,slavin2022adapting} \textcolor{black}{and a variant of Nesterov's `worst function in the world' \cite{Kozak,nesterov2013} ($f_{N,\lambda,s}$). These functions are defined as}
\begin{align*}
    f_{ms}(x) &= \sum_{i=1}^s x_{\pi(i)}^2 \quad \text{ where } |x_{\pi(1)}| \geq |x_{\pi(2)}| \geq \ldots \geq |x_{\pi(n)}|, \\
    f_{N,\lambda,s}(x) &= \frac{\lambda}{8}\left(x_1^2 + \sum_{i=1}^s\left(x_{i} - x_{i+1}\right)^2 + x_s^2  \right) - \frac{\lambda}{4}x_1.
\end{align*}
The function $f_{ms}$ is challenging as, while $\nabla f_{ms}$ is $p$-compressible (in fact, sparse), the location of the largest-in-magnitude entries $\nabla f_{ms}$ changes frequently. Moreover, $\nabla f_{ms}$ is not Lipschitz continuous. \textcolor{black}{The function $f_{N,\lambda,s}$ is a variant of a well-known pathological example used to prove tightness of convergence rate bounds in convex optimization. Note that $\nabla f_{N,\lambda,s}$ is Lipschitz continuous with Lipschitz constant $\lambda/2$.} For both functions, we take $n$, the dimension of the function, to be $1000$ and $s=30$. For $f_{N,\lambda,s}$ we take $\lambda = 8$ as in \cite{Kozak}.


\paragraph{Parameters.} \textcolor{black}{We used the default parameters given in the code for {\tt DFQRM}, as these worked well. For {\tt DS} we used the parameters given in Section 4 of \cite{Roberts} for the Gaussian polling direction, $r=1$, setting.}\footnote{That is, at each iteration, we query $f$ at $f(x_k + z)$ and $f(x_k - z)$ where $z$ is a Gaussian random vector as this setting appears to work best.}\textcolor{black}{ That is, with notation as in \cite{Roberts}, we take $\alpha_0 = 1.0, \alpha_{\max} = 1000, \gamma_{\mathrm{inc}} = 2, \gamma_{\mathrm{dec}} = 0.5$, and we terminate whenever $\alpha_k < 10^{-6}$. For {\tt SSD} there did not appear to be a `best' choice of subspace dimension in the results of \cite{Kozak}. So, we take this to be equal to $s$ in our experiments.}\footnote{We experimented with other values, and did not see significantly different performance.}. \textcolor{black}{Three algorithms require a step-size parameter: {\tt ZORO}, {\tt adaZORO}, and {\tt SSD}, which should theoretically be equal to or proportional to the inverse of the Lipschitz constant of the objective function. In order to guarantee convergence for both $f_{ms}$ and $f_{N,\lambda, s}$, we set the step size to the more pessimistic value\footnote{Although $f_{ms}$ is not everwhere differentiable, it has a Lipshitz constant of $2$ wherever it is differentiable.} of $1/\lambda$. Finally, both {\tt ZORO} and {\tt adaZORO} are given the exact sparsity. To make the setting more challenging for {\tt ZORO-FA}, we deliberately provide it with an underestimate of the true sparsity as $s_0$, namely $20$. See Table~\ref{table: params_sparse_benchmarking} for the values of the major parameters used.} In all trials $x_0$ is a random draw from the multivariate normal distribution with mean $0$ and covariance $10I$. Each algorithm is given a query budget of $350(n+1)$.

\begin{table}[!]
\def\ROWCOLOR{black!10!white}
\centering
    \begin{tabular}{l c c c c c c c c}
    \toprule
       Algorithm & $b$ & step size & $s_0\dagger$ & $\epsilon$ & $\theta\ddagger$ &  $\sigma_0\lozenge$ & $\sigma_{\min}^\star$ & samp. radius \\
    \midrule
    \rowcolor{\ROWCOLOR}
    {\tt DFQRM} & - & - & - & $10^{-5}$ & $0$ & $1$ & $0.01$  & - \\
    {\tt ZORO} & $1$ & $1/8$ & $30$ & - & - & - & - & $10^{-4}$ \\
    \rowcolor{\ROWCOLOR}
    {\tt adaZORO} & $0.5$ & $1/8$ & $30$ & - & - & - & - & $10^{-4}$ \\
    {\tt ZORO-FA} & $1$ & - & $20$ & $10^{-5}$ & $0.25$ & $2.5$ & - & - \\
    \rowcolor{\ROWCOLOR}
    {\tt Nelder-Mead} & - & - & - & - & - & - & $0.001$ & -\\
    \textcolor{black}{{\tt SSD}} & - & $\frac{s_0}{8n}$ & $30$ & - & - & - & - & $10^{-4}$ \\
    \rowcolor{\ROWCOLOR}
    \textcolor{black}{{\tt DS-RS}} & - & - & $1$ & - & - & 1 & $10^{-6}$ & - \\
    \bottomrule
    \end{tabular}
    \caption{Parameters used in sparse benchmarking experiments of Section~\ref{sec:Sparse_Benchmarking}. $\dagger$: Here $s_0$ refers to the target sparsity in {\tt ZORO}, the initial target sparsity of {\tt adaZORO} and {\tt ZORO-FA}, and the subspace dimension of {\tt SSD} and {\tt DS}. $\ddagger$: In {\tt DFQRM}, $\theta$ has a slightly different meaning. $\lozenge$: For {\tt DS}, this is the initial step size ($\alpha_0$ in \cite{Roberts}). $\star$: In {\tt Nelder-Mead} this is the minimum simplex size; in {\tt DS-RS} this is threshold for $\alpha_k$ at which we terminate.}
    \label{table: params_sparse_benchmarking}
\end{table}

\paragraph{Results.} Figure~\ref{fig:max-s-results} displays typical trajectories for both functions. We present just a single trajectory as we observed very little variation between runs. \textcolor{black}{As is clear, {\tt ZORO-FA} substantially outperforms all other algorithms. Although {\tt ZORO} and {\tt adaZORO} descend quickly at the beginning---likely due to the fact they are given the exact sparsity as input---their progress quickly stalls out as they are unable to decrease their step size or sampling radius. Finally, we note that all algorithms save {\tt Nelder-Mead} had similar run times---around $10$--$50$ seconds. {\tt Nelder-Mead} typically took two orders of magnitude more time to run.}


\begin{figure}
    \centering
    \includegraphics[width=0.48\textwidth]{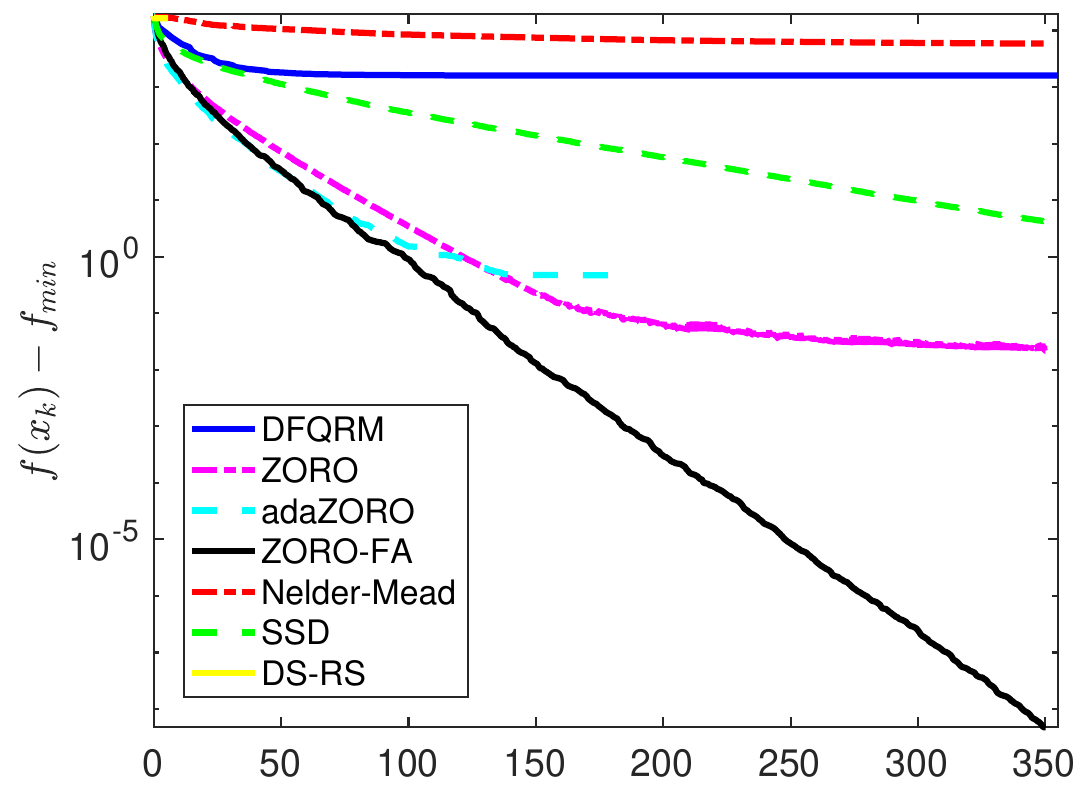}
    \includegraphics[width=0.46\textwidth]{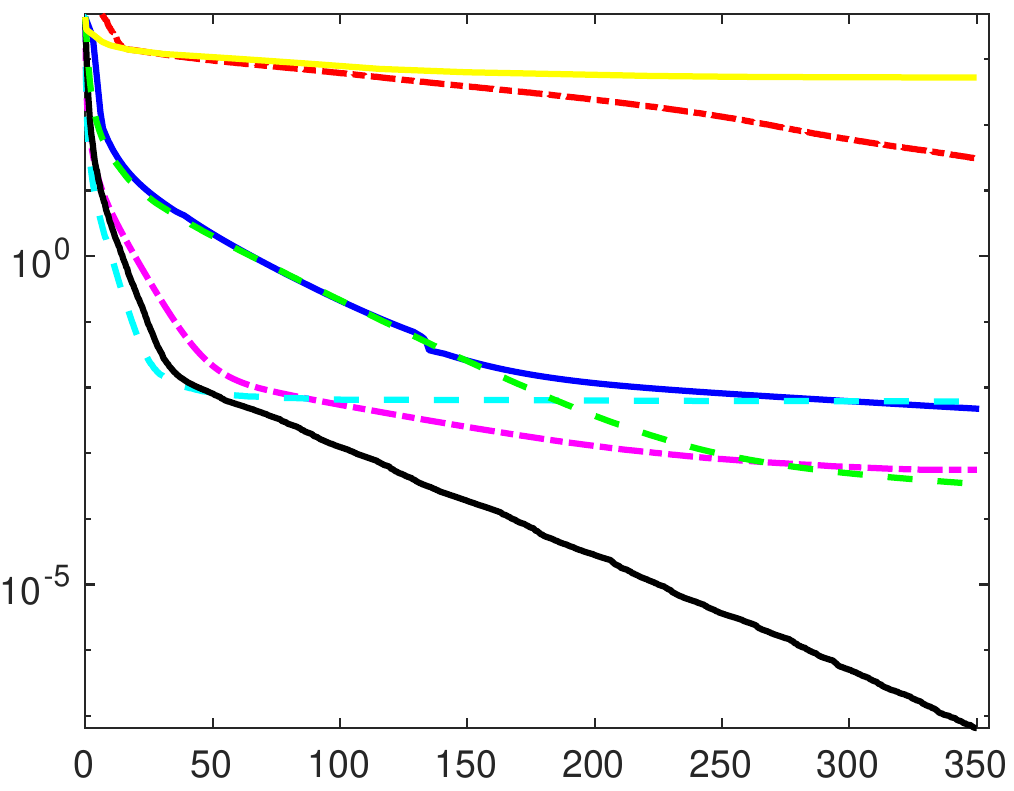}
    \caption{\textcolor{black}{{\bf Left:} Typical trajectories on $f_{ms}$. {\tt adaZORO} fails when it is unable to find a gradient approximation satisfying its acceptance criterion. {\tt DS-RS} fails when $\alpha_k$ fall below $10^{-6}$. {\bf Right:} Typical trajectories on $f_{N,\lambda,s}$. Note the $x$-axis displays the cumulative number of queries divided by $n+1$. }}
    \label{fig:max-s-results}
\end{figure}



\subsection{Mor\'{e}-Garbow-Hillstrom Test Functions}
\label{sec:MGH}

Next, we benchmark {\tt ZORO-FA} on a subset of the well-known Mor\'{e}-Garbow-Hillstrom (MGH) suite of test problems \cite{more1981testing}. \textcolor{black}{For simplicity, we do not include {\tt SSD} or {\tt DS-RS} in this experiment.}

\paragraph{Functions.} We select the 18 problems with variable dimensions $n$ and considered three values of $n$: $n=100$, $n=500$ and $n=1000$. We synthetically increase the number of functions to 36 by toggling the initial point $x_0 = 10^s\tilde{x}_0$, for $s \in \{0,1\}$ and $\tilde{x}_0$ the default initial point associated with the MGH function in question.

\paragraph{Parameters.} We tuned the parameters of each algorithm to the best of our ability but use the same parameter values for all test functions, for all values of $n$, see Table~\ref{table: params_MGH_benchmarking} for precise values. 

\paragraph{Data Profiles.} We present the results as data profiles \cite{more2009benchmarking} for various tolerances $\tau$. Let $\mathcal{P}$ denote the test problem set. A point $(\alpha, y)$ on a curve in a data profile indicates that a fraction of $y|\mathcal{P}|$ problems were solved, to tolerance $\tau$, given $(n+1)\times \alpha$ function evaluations. We say an algorithm solves a problem to tolerance $\tau$ if it finds an iterate $x_k$ satisfying
\begin{equation}
  f(x_k) \leq f_{L} + \tau \left(f(x_0) - f_L\right),  
\end{equation}
where $f_L$ is the smallest objective function values attained by any algorithm on the problem.

\begin{table}[ht!]
\def\ROWCOLOR{black!10!white}
\centering
    \begin{tabular}{l c c c c c c c c}
    \toprule
       Algorithm & $b$ & step size & $s_0\dagger$ & $\epsilon$ & $\theta$ &  $\sigma_0$ & $\sigma_{\min}^{\star}$ & samp. radius \\
    \midrule
    \rowcolor{\ROWCOLOR}
    {\tt DFQRM} & - & - & - & $10^{-5}$ & $0\ddagger$ & $1$ & $0.01$  & - \\
    {\tt ZORO} & $1$ & $\frac{0.005}{s_0\log n}$ & $0.1n$ & - & - & - & - & $5\times 10^{-4}$ \\
    \rowcolor{\ROWCOLOR}
    {\tt adaZORO} & $0.5$ & $\frac{0.005}{s_0\log n}$ & $0.1n$ & - & - & - & - & $5\times 10^{-4}$ \\
    {\tt ZORO-FA} & $1$ & - & $0.1n$ & $0.01$ & $0.25$ & $\frac{1}{s_0\log n}$ & - & - \\
    \rowcolor{\ROWCOLOR}
    {\tt Nelder-Mead} & - & - & - & - & - & - & $0.001$ & -\\
    \bottomrule
    \end{tabular}
    \caption{Parameters used in the Mor\'{e}-Garbow-Hillstrom benchmarking experiments of Section~\ref{sec:MGH}. $\dagger$: Here $s_0$ refers to the target sparsity in {\tt ZORO}, and the initial target sparsity of {\tt adaZORO} and {\tt ZORO-FA} (which both dynamically adjust this parameter). $n$ refers to the problem size. $\ddagger$: In {\tt DFQRM}, $\theta$ has a slightly different meaning. $\star$: In {\tt Nelder-Mead} this is the minimum simplex size.}
    \label{table: params_MGH_benchmarking}
\end{table}

\paragraph{Results.} The results are shown in Figure~\ref{fig:MGH-data-profile}. We observe that {\tt ZORO-FA} is quickly able to solve a subset of MGH problems, as shown by the left-hand side of each data profile. We found that {\tt adaZORO} and {\tt ZORO} struggled in all settings. We suspect this is because (i) Only some MGH test functions exhibit gradient sparsity, and (ii) The Lipschitz constants, and consequently the ideal step sizes, vary greatly among MGH functions. Thus, algorithms which are able to adaptively select their step size have an innate advantage.  

\begin{figure}[h!]
    \centering
\includegraphics[width=0.32\textwidth]{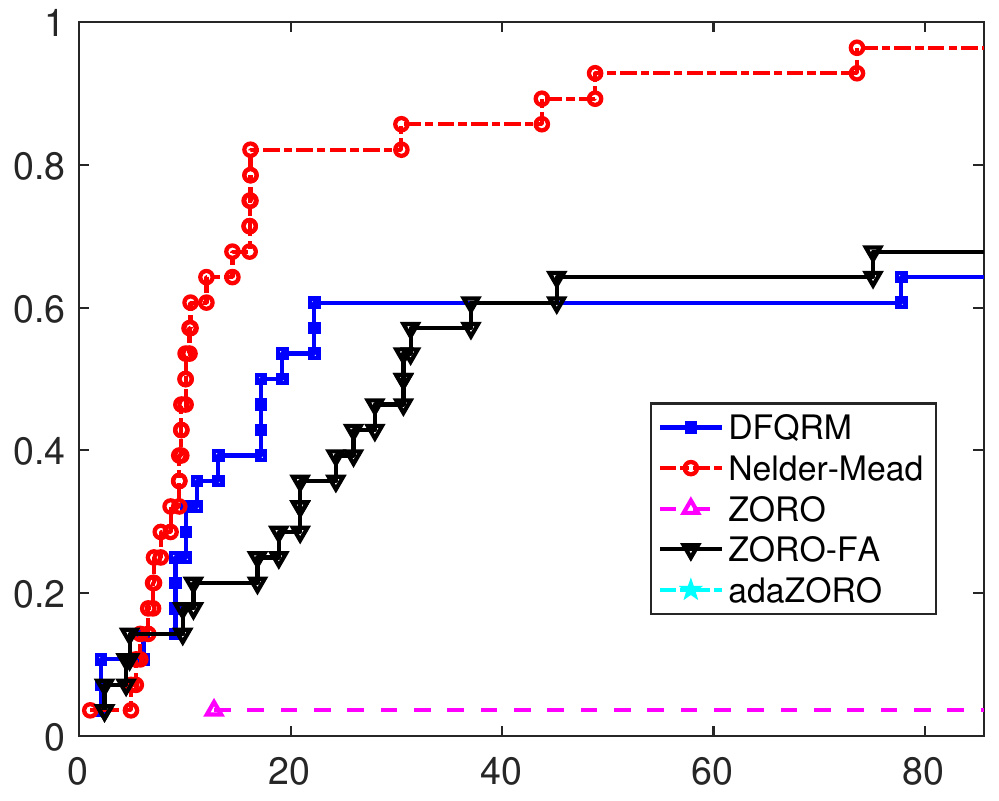} \hfill
\includegraphics[width=0.32\textwidth]{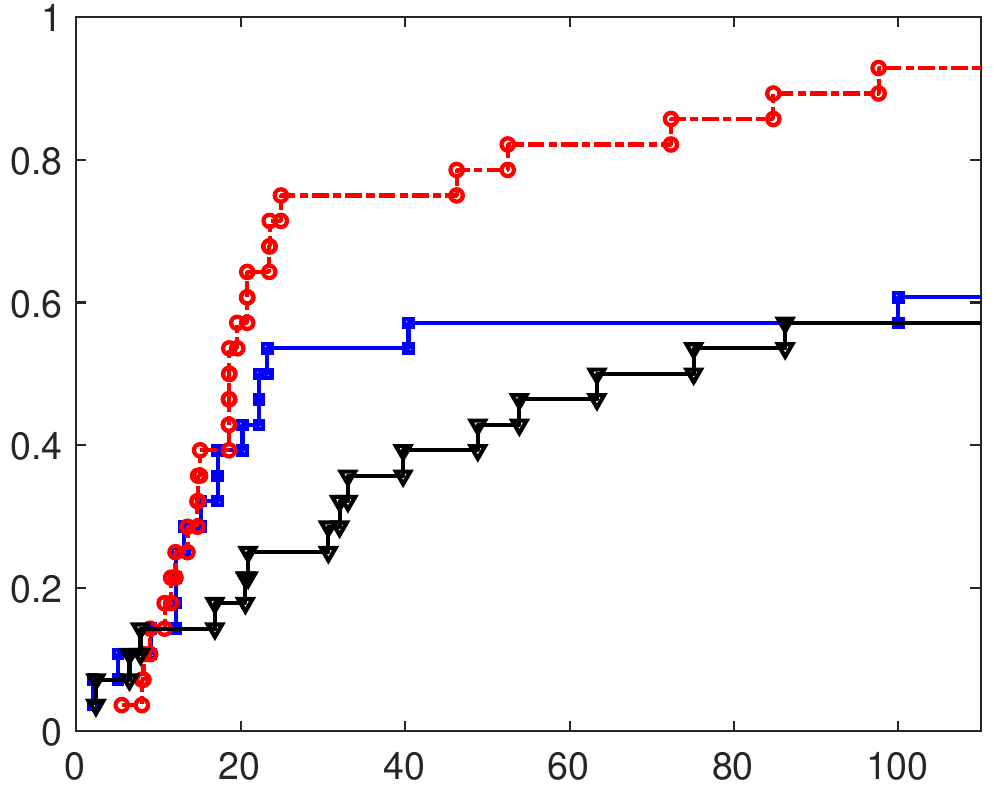} \hfill
\includegraphics[width=0.32\textwidth]{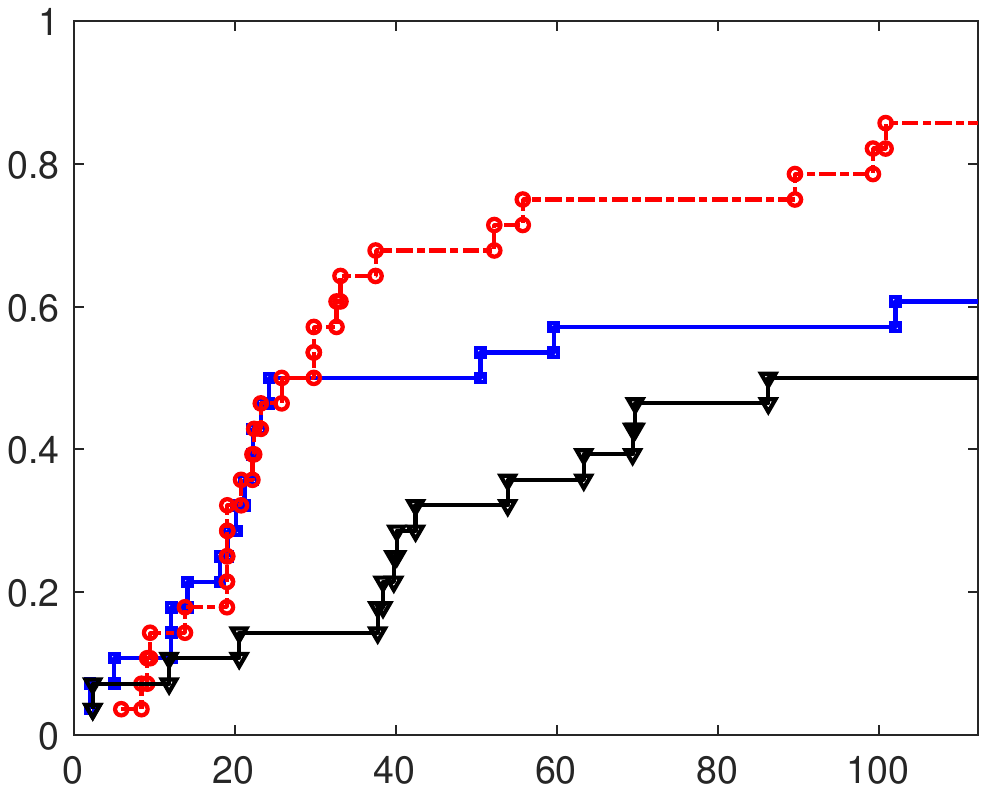} \\
 \includegraphics[width=0.32\textwidth]{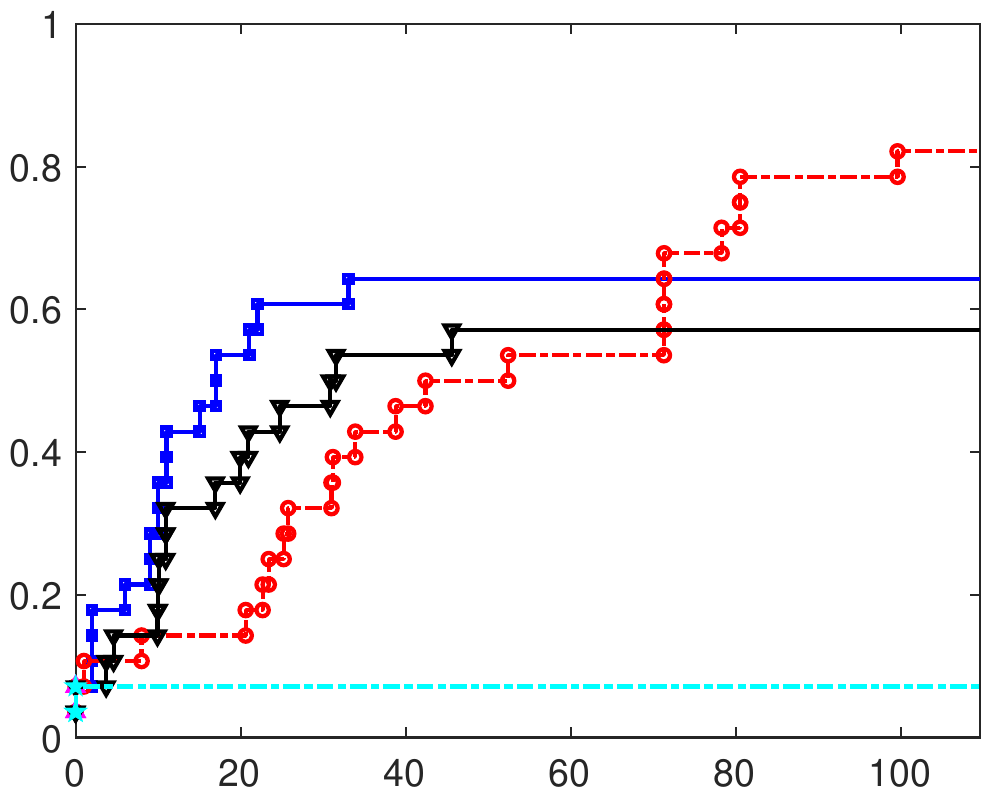} \hfill
 \includegraphics[width=0.32\textwidth]{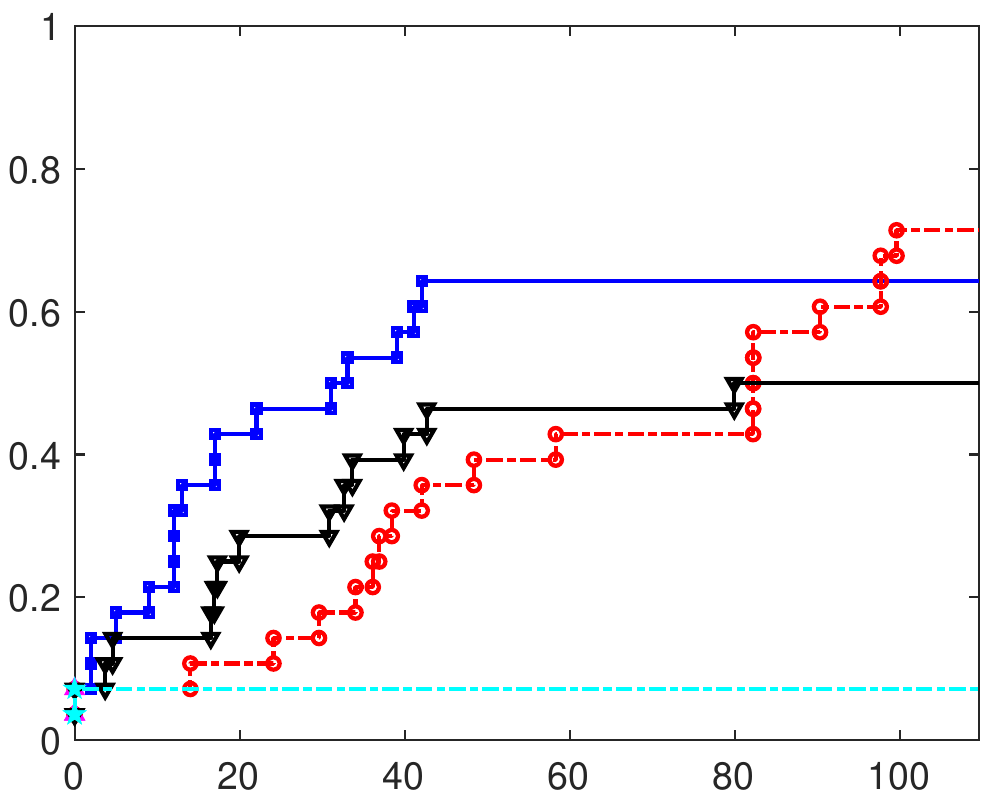} \hfill
 \includegraphics[width=0.32\textwidth]{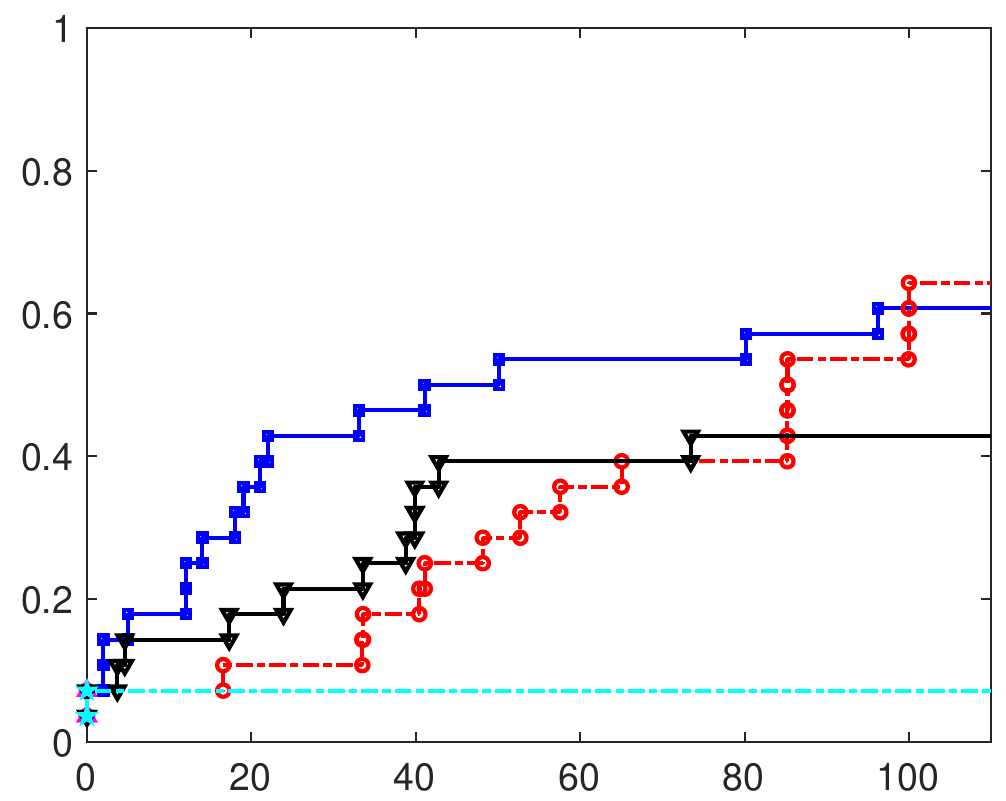}
 \includegraphics[width=0.32\textwidth]{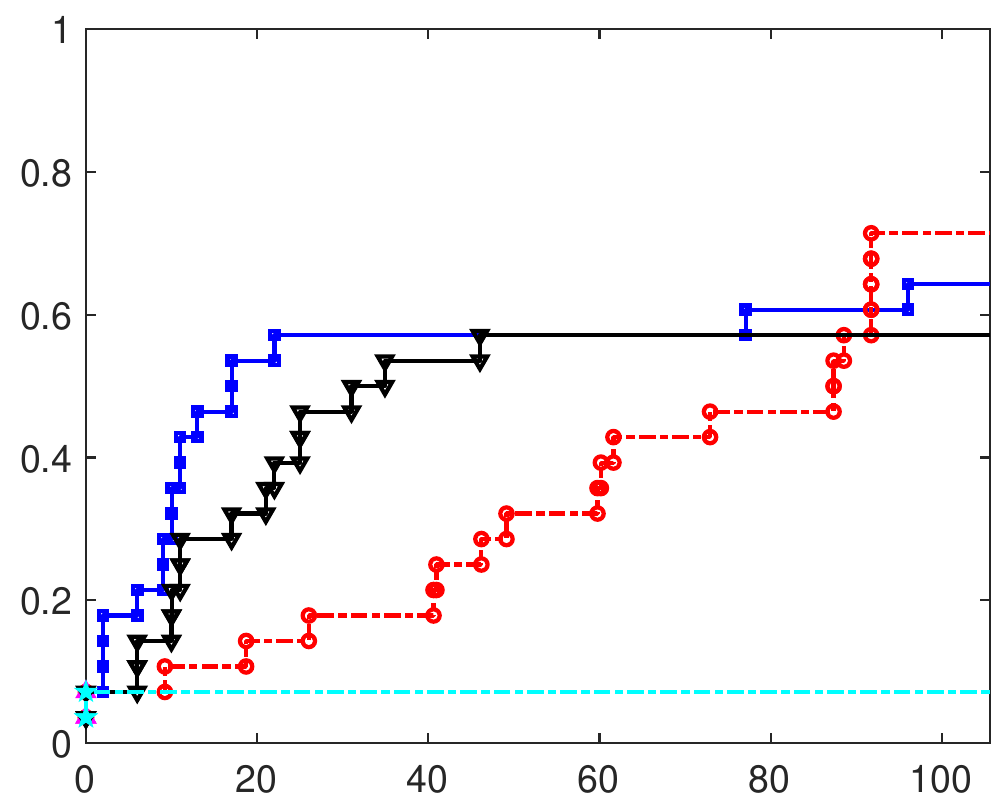} \hfill
 \includegraphics[width=0.32\textwidth]{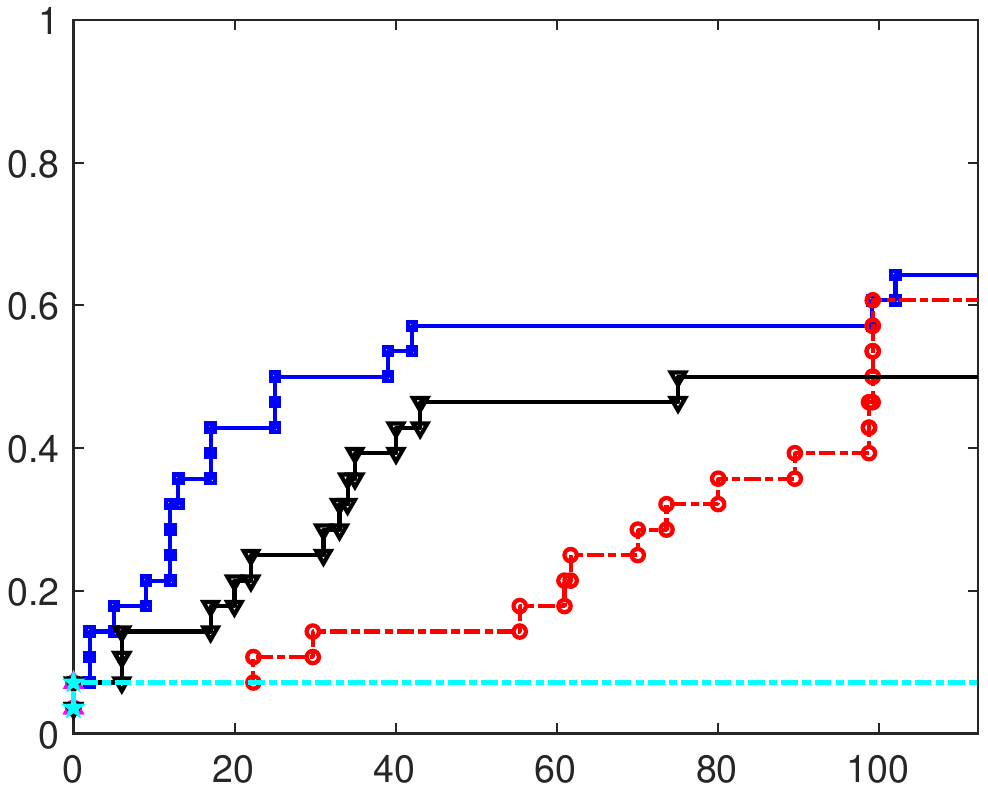} \hfill
 \includegraphics[width=0.32\textwidth]{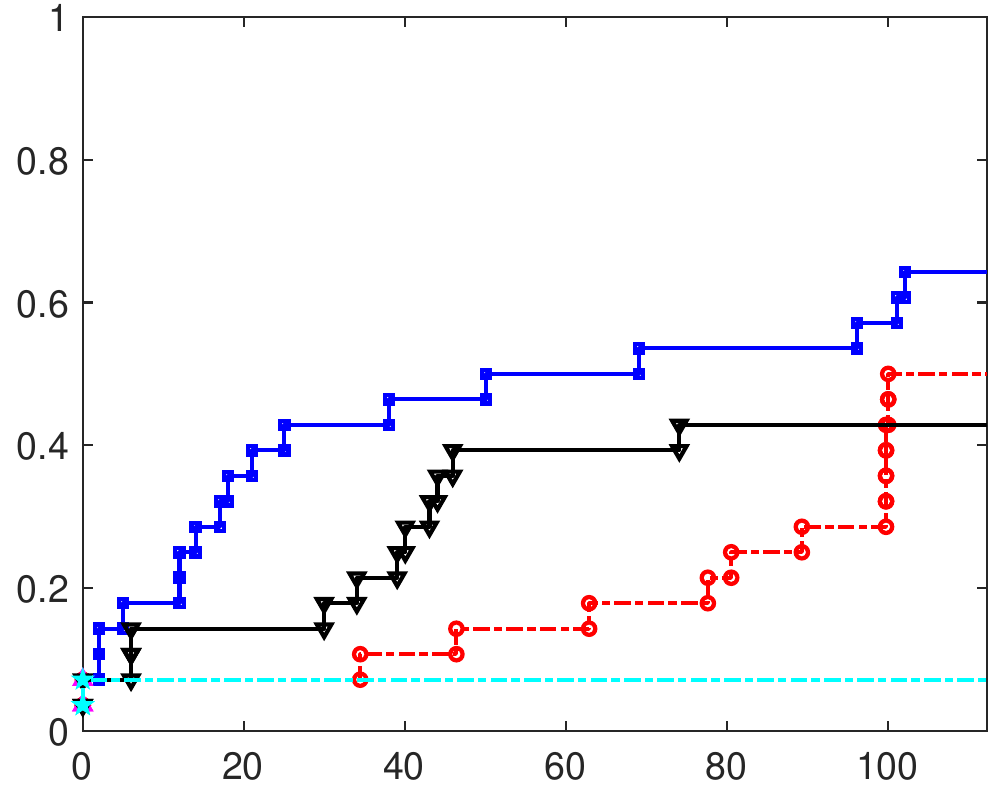}
    \caption{Data profiles for various tolerances ($\tau$) and problem dimensions ($n$).  {\bf Top Row:} $n=100$. {\bf Middle Row:} $n=500$. {\bf Bottom Row:} $n=1000$. {\bf Left Column:} $\tau = 10^{-1}$. {\bf Middle Column:} $\tau = 10^{-2}$. {\bf Right Column:} $\tau = 10^{-3}$. In the top row, plots for {\tt ZORO} and {\tt adaZORO} are absent as they did not solve any problems to the required tolerance.}
    \label{fig:MGH-data-profile}
\end{figure}

\subsection{Discussion}

\textcolor{black}{On functions exhibiting gradient sparsity
{\tt ZORO-FA} outperforms existing benchmarks. Even though {\tt ZORO} and {\tt adaZORO} are given the exact gradient sparsity parameter, the fact that they cannot decrease their step-size or sampling radius means that their progress eventually plateaus.} For the MGH test functions {\tt ZORO} and {\tt adaZORO} struggle, 
while {\tt Nelder-Mead}, {\tt DFQRM}, and {\tt ZORO-FA} perform well. It was surprising, at least to us, that {\tt Nelder-Mead} performed so well. We note however that this performance came at high computational expense, with {\tt Nelder-Mead} typically taking 10--100 times longer to run.

\subsection{Hidden Gradient Compressibility}
In contrast to Section~\ref{sec:Sparse_Benchmarking}, the MGH functions of Section~\ref{sec:MGH} {\em are not} explicitly designed to possess sparse or compressible gradients. Nonetheless we observe that {\tt ZORO-FA} works well, and is frequently able to make progress with $s_k \ll n$. To investigate this further, we plot the magnitudes of the entries of $\nabla f(x)$, sorted from largest to smallest, for two MGH test functions (with $n=500$) at $20$ randomly selected $x$. Letting $|\nabla f(x)|_{(j)}$ denote the $j$-th largest-in-magnitude component of $\nabla f(x)$, we plot the mean of $|\nabla f(x)|_{(j)}$ over all $20$ values of $x$, as well as the minimum and maximum values of $|\nabla f(x)|_{(j)}$ observed.

\textcolor{black}{We also plot a trajectory of {\tt ZORO-FA} applied to these two functions, as well as the sparsity levels found at each iteration ($s_k$). The results are displayed in Figure~\ref{fig:MGH-gradient-compressibility}. As is clear, some MGH test functions (e.g., {\tt rosex}) exhibit extreme gradient compressibility, while others (e.g., {\tt trig}) do not. As suggested above, {\tt ZORO-FA} can indeed make progress using $s_k$ smaller than $n$. It is interesting that {\tt ZORO-FA} uses lower values of $s_k$ for the function ({\tt trig}) whose gradients appear less compressible. We suspect this might be due to the interplay of the sparsity and Lipschitz continuity of the gradients. As {\tt ZORO-FA} searches for a suitable sparsity level ($s_j$) and Lipschitz constant ($\sigma_j$) simultaneously, it might `miss out' on gradient sparsity when the Lipschitz constant is large. This further emphasizes the need for further research into optimization methods which can take advantage of the presence of gradient sparsity, yet are robust to its absence.}  

\begin{figure}[h]
    \centering
    \includegraphics[width=0.46\textwidth]{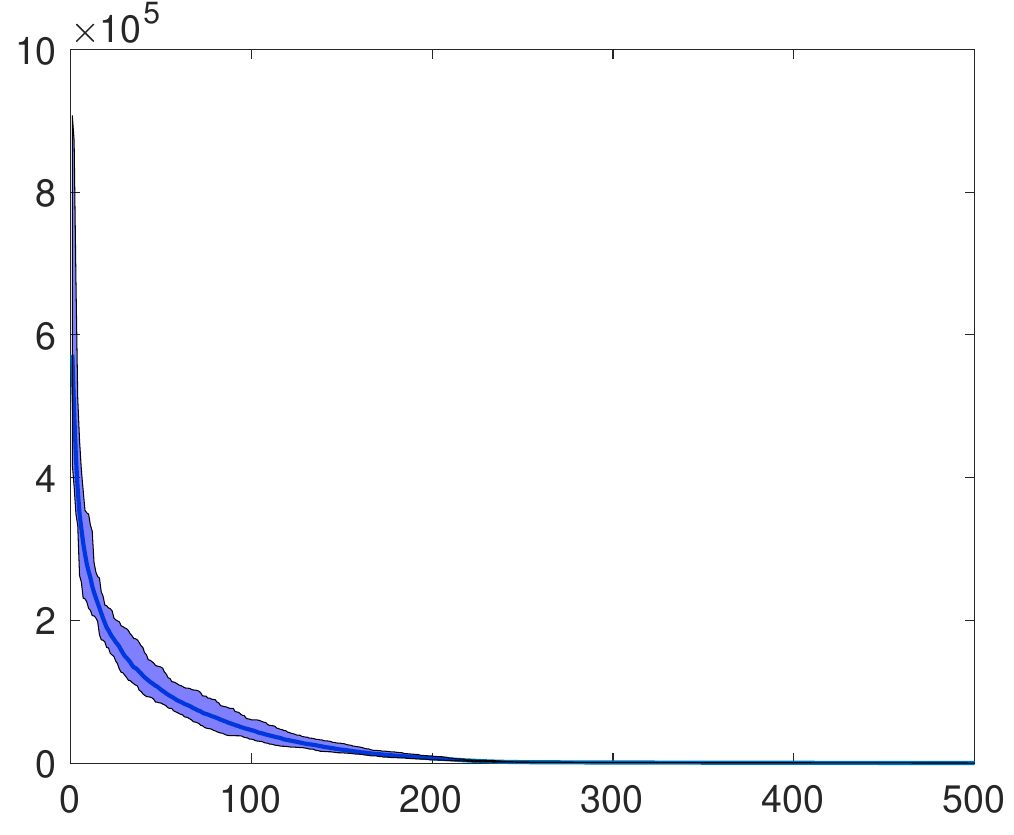} \hfill
    \includegraphics[width=0.46\textwidth]{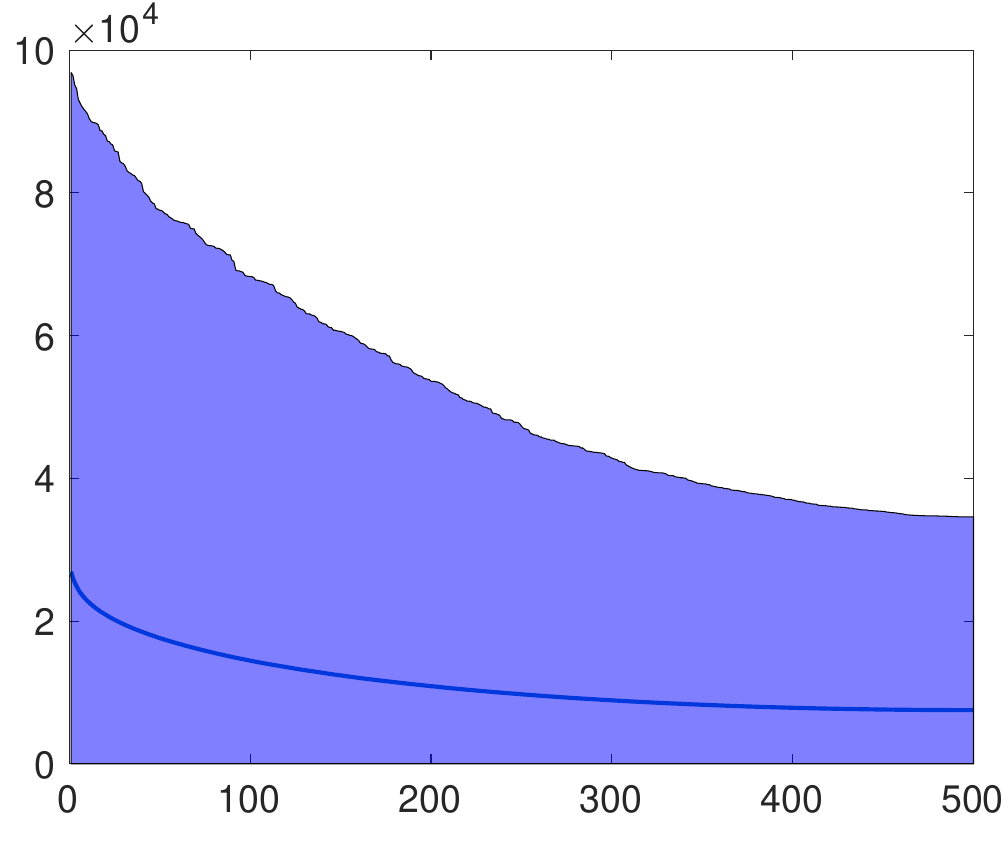} \\
    \includegraphics[width=0.46\textwidth]{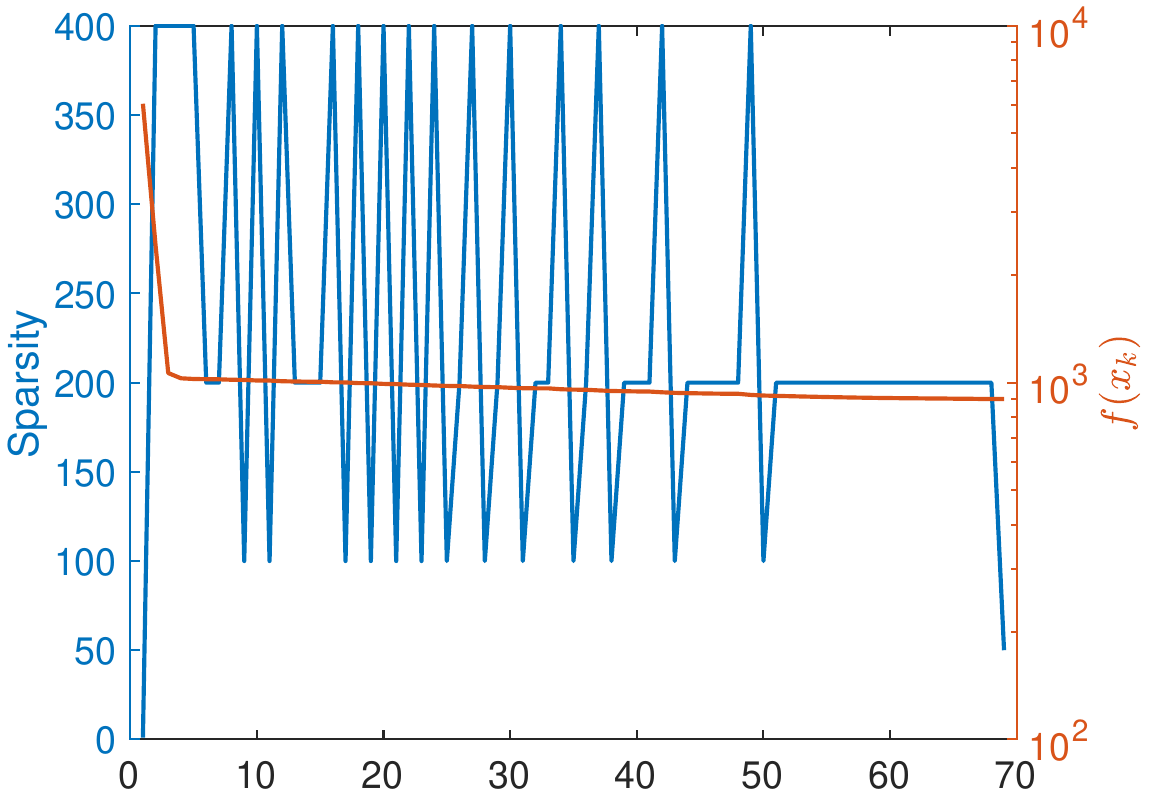} \hfill
    \includegraphics[width=0.46\textwidth]{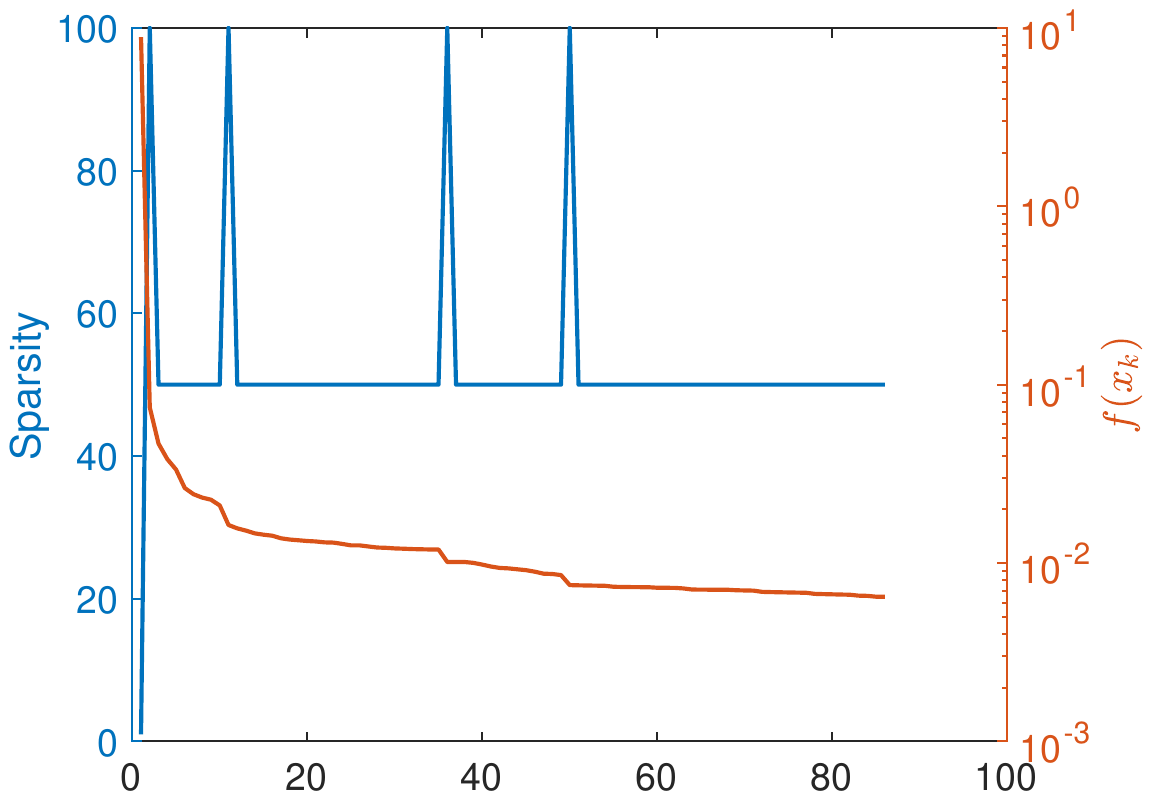}
    \caption{{\bf Top row:} Gradient magnitude profiles for two MGH functions: {\tt rosex} ({\bf Left}), and {\tt trig} ({\bf Right}). The solid line represents the mean (over 20 trials), while the shading represents the minimum and maximum values of gradient component magnitudes observed. {\bf Bottom row:} The successful sparsity level for {\tt ZORO-FA} (left axes) as well as the objective function values (right axes) for {\tt rosex} ({\bf Left}), and {\tt trig} ({\bf Right}).}
    \label{fig:MGH-gradient-compressibility}
\end{figure}

\section{Conclusion}

In this paper, we presented a fully adaptive variant of {\tt ZORO}, originally proposed in \cite{cai2022zeroth}, for derivative-free minimization of a smooth functions with $n$ variables and $L$-Lipschitz continuous gradient. The new method, called {\tt ZORO-FA}, does not require the knowledge of the Lipschitz constant $L$ or the effective sparsity level (when the gradients are compressible). At each iteration, {\tt ZORO-FA} attempts to exploit a possible compressibility of gradient to compute a suitable gradient approximation \textit{using fewer than} $\mathcal{O}\left(n\right)$ function evaluations. When the corresponding trial points fail to ensure a functional decrease of $\mathcal{O}\left(\epsilon^{2}\right)$, the method computes gradient approximations via forward finite-differences, which guarantee that the trial points will yield the desired functional decrease. With this safeguard procedure, we proved that, with probability one, {\tt ZORO-FA} needs no more than $\mathcal{O}\left(n^2\epsilon^{-2}\right)$ function evaluations to find an $\epsilon$-approximate stationary point of the objective function. When the gradient vectors of the objective function are $p$-compressible for some $p\in (0,1)$, we proved that, with high probability, {\tt ZORO-FA} has an improved worst-case complexity of $\mathcal{O}\left(s(\theta,p)\ln(n)\epsilon^{-2}\right)$ function evaluations, where $\theta\in (0,1/2)$ is a user-defined parameter that controls the relative error of the gradient approximations, and $s(\theta,p)$ is the effective sparsity level as defined in Corollary~\ref{cor:2.1}.
Our preliminary numerical results indicate that {\tt ZORO-FA} is able to exploit the presence of gradient compressibility, while being robust towards its absence.

\section*{Acknowledgements} 
\textcolor{black}{We are very grateful to the associate editor and the two anonymous reviewers for their valuable comments and suggestions, which helped us improve this work.}

\section*{Declarations}
\begin{itemize}
\item Conflict of interest: The authors have no conflicts of interest to declare that are relevant to the content of this article.
\item Data availability: The Mor\'{e}-Garbow-Hillstrom benchmarking functions used in the experiments of Section~\ref{sec:MGH} can be downloaded at \url{https://arnold-neumaier.at/glopt/test.html}
\end{itemize}


\begin{thebibliography}{10}

\bibitem{balasubramanian2018zeroth} Balasubramanian, K., Ghadimi, S.: Zeroth-order (non)-convex stochastic optimization via conditional gradient and gradient updates. Advances in Neural Information Processing Systems 31 (2018)

\bibitem{Bandeira}
Bandeira, A.S., Scheinberg, K., Vicente, L.N.: Computation of sparse low degree interpolating polynomials and their application to derivative-free optimization. Mathematical Programming 134, 223--257 (2012)

\bibitem{baraniuk2008simple} Baraniuk, R., Davenport, M., DeVore, R., Wakin, M.: A simple proof of the restricted isometry property for random matrices. Constructive approximation 28, 253--263 (2008)

\bibitem{Bergou}
Bergou, E.H., Gorbunov, E., Richt\'arik, P.: Stochastic three points method for unconstrained smooth optimization. SIAM Journal on Optimization 30 (2020)

\bibitem{bergstra2012random} Bergstra, J., Bengio, Y.: Random search for hyper-parameter optimization. Journal of Machine Learning Research 1, 281--305 (2012)

\bibitem{cai2021zeroth} Cai, H., Lou, Y., McKenzie, D., Yin, W.: A zeroth-order block coordinate descent algorithm for huge-scale black-box optimization. PMLR, 1193--1203 (2021)

\bibitem{cai2022one} Cai, H., McKenzie, D., Yin, W., Zhang, Z.: A one-bit, comparison-based gradient estimator. Applied and Computational Harmonic Analysis 60, 242--266 (2022)

\bibitem{cai2022zeroth} Cai, H., Mckenzie, D., Yin, W., Zhang, Z.: Zeroth-order regularized optimization (zoro): Approximately sparse gradients and adaptive sampling. SIAM Journal on Optimization 32, 687--714 (2022)

\bibitem{Cartis}
Cartis, C., Roberts, L.: Scalable subspace methods for derivative-free nonlinear least-squares optimization. Mathematical Programming 199, 461--524 (2023)

\bibitem{cartis2022dimensionality}
Cartis, C., Otemissov, A.: A dimensionality reduction technique for unconstrained global optimization of functions with low effective dimensionality. Information and Inference: A Journal of the IMA, 11(1), 167-201.

\bibitem{cartis2023bound}
Cartis, C., Massart, E., Otemissov, A.: Bound-constrained global optimization of functions with low effective dimensionality using multiple random embeddings. Mathematical Programming, 198(1), 997-1058.

\bibitem{chen2023deepzero} Chen, A., Zhang, Y., Jia, J., Diffenderfer, J., Liu, J., Parasyris, K., Zhang, Y., Zhang, Z., Kailkhura, B., Liu, S.: Deepzero: Scaling up zeroth-order optimization for deep model training. arXiv preprint arXiv:2310.02025 (2023)

\bibitem{choromanski2020provably} Choramanski, K, Pacchiano, A., Parker-Harbor, J., Tang, Y., Jain, D., Yang, Y., Iscen, A., Hsu, J., Sindhwani, V.: Provably robust blackbox optimization for reinforcement learning. Conference on Robot Learning, 683--696 (2020)

\bibitem{foucart2012sparse} Foucart, S.: Sparse recovery algorithms: sufficient conditions in terms of restricted isometry constants. Approximation Theory XIII: San Antonio 2010. Springer, 2012.

\bibitem{grapiglia2023quadratic} Grapiglia, G.N.: Quadratic regularization methods with finite-difference gradient approximations. Computational Optimization and Applications 85, 683--703 (2023)

\bibitem{grapiglia2024quadratic} Grapiglia, G.N.: Worst-case evaluation complexity of a derivative-free quadratic regularization method. Optimization Letters 18, 195--213 (2024)

\bibitem{kim2021curvature} Kim, B., Cai, H., McKenzie, D., Yin, W.: Curvature-aware derivative-free optimization. arXiv preprint arXiv:2109.13391 (2021)

\bibitem{Kozak}
Kozak, D., Becker, S., Doostan, A., Tenorio, L.: A stochastic subspace approach to gradient-free optimization in high dimensions. Computational Optimization and Applications 79, 339--368 (2021)

\bibitem{kozak2023}
Kozak, D., Molinari, C., Rosasco, L., Tenorio, L., Villa, S.: Zeroth-order optimization with orthogonal random directions. Mathematical Programming, 199(1), 1179-1219 (2023)

\bibitem{liu2024sparse} Liu, Y., Zhu, Z., Gong, C., Cheng, M., Hsieh, C-J., You, Y.: Sparse MeZO: Less Parameters for Better Performance in Zeroth-Order LLM Fine-Tuning. arXiv preprint arXiv:2402.15751 (2024)

\bibitem{malladi2024fine} Malladi, S., Gao, T., Nichani, E., Damian, A., Lee, J.D., Chen, D., Arora, S.: Fine-tuning language models with just forward passes. Advances in Neural Information Processing Systems 36 (2024) 

\bibitem{mania2018simple} Mania, H., Guy, A., Recht, B.: Simple random search of static linear policies is competitive for reinforcement learning. Proceedings of the 32nd International Conference on Neural Information Processing Systems, 1805--1814 (2018)

\bibitem{more1981testing} Mor\'e, J.J., Garbow, B.S., Hillstrom, K.E.: Testing unconstrained optimization software. ACM Transactions on Mathematical Software 7, 17--41 (1981)

\bibitem{more2009benchmarking} Mor{\'e}, J.J., Wild, S.M.: Benchmarking derivative-free optimization algorithms. SIAM Journal on Optimization 20, 172--191 (2009)

\bibitem{Nakamura2017} Nakamura, N., Seepaul, J., Kadane, J.B., Reeja-Jayan, B.: Design for low-temperature microwave-assisted crystallization of ceramic thin films. Applied Stochastic Models in Business and Industry 33, 314--321 (2017)

\bibitem{needell2009cosamp} Needell, D., Tropp, J.A.: CoSaMP: Iterative signal recovery from incomplete and inaccurate samples. Applied and computational harmonic analysis 26, 301--321 (2009)

\bibitem{NM} Nelder, J.A., Mead, R.: A simplex method for function minimization. The Computer Journal 7, 308-313 (1965)

\bibitem{Nesterov}
Nesterov, Yu., Spokoiny, V.: Random Gradient-Free Minimization of Convex Functions. Foundations of Computational Mathematics 17, 527--566 (2017)

\bibitem{nesterov2013}
Nesterov, Yu.: Introductory Lectures on Convex Optimization: A Basic Course. Vol 87. Spring, Berlin (2013)

\bibitem{powell2006newuoa} Powell, M.J.D.: The NEWUOA software for unconstrained optimization without derivatives. Large-scale nonlinear optimization, 255--297 (2006)

\bibitem{Roberts}
Roberts, L., Royer, C.W.: Direct search based on probabilistic descent in reduced spaces. SIAM Journal on Optimization 33 (2023)

\bibitem{salimans2017evolution} Salimans, T., Ho, J., Chen, X., Sidor, S., Sutskever, I.: Evolution strategies as a scalable alternative to reinforcement learning. arXiv preprint arXiv:1703.03864 (2017)

\bibitem{slavin2022adapting} Slavin, I., McKenzie, D.: Adapting zeroth order algorithms for comparison-based optimization. arXiv preprint arXiv:2210.05824 (2022)

\bibitem{shi2023numerical} Shi, H.M., Xuan, M.Q., Oztoprak, F., Nocedal, J.: On the numerical performance of finite-difference-based methods for derivative-free optimization. Optimization Methods and Software 38, 289--311 (2023)

\bibitem{tett2022calibrating} Tett, S., Gregory, J., Freychet, N., Cartis, C., Mineter, M., Roberts, L.: Calibrating Climate Models-what observations matter? EGU General Assembly Conference Abstracts (2022)

\bibitem{wang2016bayesian}
Wang, Z., Hutter, F., Zoghi, M., Matheson, D., De Feitas, N.: Bayesian optimization in a billion dimensions via random embeddings. Journal of Artificial Intelligence Research, 55, 361-387.

\bibitem{wang2018stochastic} Wang, Y., Du, S., Balakrishnan, S., Singh, A.: Stochastic zeroth-order optimization in high dimensions. PMLR, 1356--1365 (2018) 

\bibitem{Yue2023zeroth}
Yue, P., Yang, L., Fang, C., Lin, Z.: Zeroth-order optimization with weak dimension dependency. In The Thirty Sixth Annual Conference on Learning Theory. PMLR, 4429--4472 (2023).

\bibitem{Zhang_2023} Zhang, Z.: PRIMA: Reference Implementation for Powell's Methods with Modernization and Amelioration. Available at \url{http://www.libprima.net, DOI: 10.5281/zenodo.8052654} (2023)

\bibitem{zhang2024revisiting} Zhang, Y., Li, P., Hong, J., Li, J., Zhang, Y., Zheng, W., Chen, P-Y., Lee, J.D., Yin, W., Hong, M., Wang, Z., Liu, S., Chen, T.: Revisiting Zeroth-Order Optimization for Memory-Efficient LLM Fine-Tuning: A Benchmark. arXiv preprint arXiv:2402.11592 (2024)

\bibitem{qiu2024gradient} R. Qiu, H. Tong: Gradient Compressed Sensing: A Query-Efficient Gradient Estimator for High-Dimensional Zeroth-Order Optimization. Proceedings of the Forty-first International Conference on Machine Learning (2024)


\end{thebibliography}
\end{document}